\documentclass[numbers=enddot,12pt,final,onecolumn,notitlepage]{scrartcl}%
\usepackage[headsepline,footsepline,manualmark]{scrlayer-scrpage}
\usepackage{amssymb}
\usepackage{amsmath}
\usepackage{amsthm}
\usepackage{framed}
\usepackage{comment}
\usepackage{color}
\usepackage[breaklinks=True]{hyperref}
\usepackage[sc]{mathpazo}
\usepackage[T1]{fontenc}
\usepackage{needspace}
\providecommand{\U}[1]{\protect\rule{.1in}{.1in}}
\theoremstyle{definition}
\newtheorem{theo}{Theorem}[section]
\newenvironment{theorem}[1][]
{\begin{theo}[#1]\begin{leftbar}}
{\end{leftbar}\end{theo}}
\newtheorem{lem}[theo]{Lemma}
\newenvironment{lemma}[1][]
{\begin{lem}[#1]\begin{leftbar}}
{\end{leftbar}\end{lem}}
\newtheorem{prop}[theo]{Proposition}
\newenvironment{proposition}[1][]
{\begin{prop}[#1]\begin{leftbar}}
{\end{leftbar}\end{prop}}
\newtheorem{defi}[theo]{Definition}
\newenvironment{definition}[1][]
{\begin{defi}[#1]\begin{leftbar}}
{\end{leftbar}\end{defi}}
\newtheorem{remk}[theo]{Remark}
\newenvironment{remark}[1][]
{\begin{remk}[#1]\begin{leftbar}}
{\end{leftbar}\end{remk}}
\newtheorem{coro}[theo]{Corollary}

\newtheorem{conv}[theo]{Convention}

\newtheorem{quest}[theo]{Question}
\newenvironment{question}[1][]
{\begin{quest}[#1]\begin{leftbar}}
{\end{leftbar}\end{quest}}
\newtheorem{warn}[theo]{Warning}
\newenvironment{warning}[1][]
{\begin{warn}[#1]\begin{leftbar}}
{\end{leftbar}\end{warn}}
\newtheorem{conj}[theo]{Conjecture}

\newtheorem{exam}[theo]{Example}
\newenvironment{example}[1][]
{\begin{exam}[#1]\begin{leftbar}}
{\end{leftbar}\end{exam}}
\newenvironment{statement}{\begin{quote}}{\end{quote}}

\let\sumnonlimits\sum
\let\prodnonlimits\prod
\let\cupnonlimits\bigcup
\let\capnonlimits\bigcap
\renewcommand{\sum}{\sumnonlimits\limits}
\renewcommand{\prod}{\prodnonlimits\limits}
\renewcommand{\bigcup}{\cupnonlimits\limits}
\renewcommand{\bigcap}{\capnonlimits\limits}
\newenvironment{verlong}{}{}
\newenvironment{vershort}{}{}
\newenvironment{noncompile}{}{}
\excludecomment{verlong}
\includecomment{vershort}
\excludecomment{noncompile}

\newcommand{\id}{\operatorname{id}}
\newcommand{\set}[1]{\left\{ #1 \right\}}

\newcommand{\tup}[1]{\left( #1 \right)}

\newcommand{\email}[1]{\href{mailto:#1}{\texttt{#1}}}
\ihead{The Dowker theorem via discrete Morse theory}
\ohead{page \thepage}
\begin{document}

\title{The Dowker theorem via discrete Morse theory}
\author{Morten Brun\footnote{Department of Mathematics, University of Bergen, Bergen, Norway. (\email{morten.brun@uib.no})},
Darij Grinberg\footnote{Department of Mathematics, Drexel University, Philadelphia, U.S.A. (\email{darijgrinberg@gmail.com})}}
\date{20 July 2024}
\maketitle
\begin{abstract}
\textbf{Abstract.}
The Dowker theorem is a classical result in the topology of finite spaces,
claiming that any binary relation between two finite spaces defines two
homotopy-equivalent complexes (the Dowker complexes).
Recently, Barmak strengthened this to a simple-homotopy-equivalence.
We reprove Barmak's result using a combinatorial argument that constructs
an explicit acyclic matching in the sense of discrete Morse theory.
\end{abstract}

\section{Introduction}

The classical Dowker theorem (in the form given by Bj\"{o}rner
\cite[Theorem 10.9]{Bjorne95}) has recently been revisited by Brun and Salbu
\cite{BruSal22}, who reproved it using what they call the \textit{rectangle
complex}.

We give a simpler proof using what we call the \textit{biclique complex}. This
proof is purely combinatorial, as it relies on discrete Morse theory instead
of homotopy theory.

We briefly discuss the relation between the biclique and rectangle complexes.

\section{Notation}

We recall the notion of a simplicial complex:

\begin{definition}
Let $W$ be a finite set. Then, a \emph{simplicial complex} (or, for short,
\emph{complex}) with ground set $W$ is defined to be a set $\Delta$ of subsets
of $W$ that is closed under taking subsets (i.e., we have $G\in\Delta$
whenever $G$ is a subset of a set $F\in\Delta$). The elements of $\Delta$ are
called the \emph{faces} of $\Delta$.
\end{definition}

Note that (unlike some authors) we do not require that each $w\in W$ belongs
to some face of $\Delta$; thus, one and the same $\Delta$ can be a complex
with several ground sets. However, the properties of complexes that we will
study (homotopical and Morse-theoretical ones) do not depend on the ground set.

Any complex $\Delta$ is a set of sets, and thus is partially ordered by
inclusion; hence, it becomes a poset. We will use the notations $\succ$ and
$\prec$ for the cover relations of this poset. Thus, two faces $A$ and $B$
satisfy $A\prec B$ (or, equivalently, $B\succ A$) if and only if $A\subseteq
B$ and $\left\vert B\setminus A\right\vert =1$.

\section{The Dowker theorem}

Let $X$ and $Y$ be two finite sets. Let $R$ be a binary relation from $X$ to
$Y$, that is, a subset of $X\times Y$. Given two elements $x\in X$ and $y\in
Y$, we write $x\ R\ y$ for $\left(  x,y\right)  \in R$.

Given two subsets $U\subseteq X$ and $V\subseteq Y$, we write $U\ \mathbf{R}%
\ V$ if and only if all $u\in U$ and all $v\in V$ satisfy $u\ R\ v$ (that is,
if and only if $U\times V\subseteq R$). This defines a binary relation
$\mathbf{R}$ from the power set of $X$ to the power set of $Y$. Note that
$\varnothing\ \mathbf{R}\ V$ holds (for vacuous reasons) for each $V\subseteq
Y$, and likewise we have $U\ \mathbf{R}\ \varnothing$ for each $U\subseteq X$.

\begin{example}
\label{exa.dowker.1}Let $X=\left\{  1,2,3,4\right\}  $ and $Y=\left\{
5,6,7,8\right\}  $, and let $R$ be the \textquotedblleft
divides\textquotedblright\ relation (i.e., we set $x\ R\ y$ if and only if
$x\mid y$ in $\mathbb{Z}$). Then, $\left\{  1,2\right\}  \ \mathbf{R}%
\ \left\{  6,8\right\}  $ but not $\left\{  1,2,4\right\}  \ \mathbf{R}%
\ \left\{  6,8\right\}  $ (since $4\nmid6$).
\end{example}

A $Y$\emph{-neighbor} of a subset $U\subseteq X$ is defined to be an element
$y\in Y$ such that $U\ \mathbf{R}\ \left\{  y\right\}  $. In other words, a
$Y$-neighbor of a subset $U\subseteq X$ means a $y\in Y$ such that all $u\in
U$ satisfy $u\ R\ y$.

\begin{vershort}
A $Y$\emph{-conic set} shall mean a subset $U\subseteq X$ that is empty or has
a $Y$-neighbor. We define $C_{X}$ to be the set of all $Y$-conic sets. This
$C_{X}$ is a simplicial complex with ground set $X$ (since a subset of a
$Y$-conic set is again a $Y$-conic set).
\end{vershort}

\begin{verlong}
A $Y$\emph{-conic set} shall mean a subset $U\subseteq X$ that is empty or has
a $Y$-neighbor. We define $C_{X}$ to be the set of all $Y$-conic sets. This
$C_{X}$ is a simplicial complex with ground set $X$ (since a subset of a
$Y$-conic set is again a $Y$-conic set\footnote{\textit{Proof.} Let $U$ be a
$Y$-conic set. Let $U^{\prime}$ be a subset of $U$. We must prove that
$U^{\prime}$ is $Y$-conic.
\par
If $U^{\prime}$ is empty, then this is obvious (by the definition of
\textquotedblleft$Y$-conic\textquotedblright). Thus, assume WLOG that
$U^{\prime}$ is nonempty. Hence, $U$ is also nonempty (since $U^{\prime}$ is a
subset of $U$). Since $U$ is $Y$-conic, this entails that $U$ has a
$Y$-neighbor (by the definition of a $Y$-conic set). Let $y\in Y$ be this
$Y$-neighbor. Thus, all $u\in U$ satisfy $u\ R\ y$ (since $y$ is a
$Y$-neighbor of $U$). Hence, all $u\in U^{\prime}$ satisfy $u\ R\ y$ as well
(since all $u\in U^{\prime}$ satisfy $u\in U^{\prime}\subseteq U$ and
therefore $u\ R\ y$ (by the preceding sentence)). In other words, $y$ is a
$Y$-neighbor of $U^{\prime}$. Hence, the set $U^{\prime}$ has a $Y$-neighbor
(namely, $y$), and thus is $Y$-conic, qed.}).
\end{verlong}

Likewise we define a simplicial complex $C_{Y}$ with ground set $Y$: An
$X$\emph{-neighbor} of a subset $V\subseteq Y$ is defined to be an element
$x\in X$ such that $\left\{  x\right\}  \ \mathbf{R}\ V$. An $X$\emph{-conic
set} shall mean a subset $V\subseteq Y$ that is empty or has an $X$-neighbor.
We define $C_{Y}$ to be the set of all $X$-conic sets.

The complexes $C_{X}$ and $C_{Y}$ will be called the \emph{left Dowker
complex} and the \emph{right Dowker complex} of the relation $R$. They have
been called $K$ and $L$ in \cite[\S 1]{Dowker52}, have been called $D\left(
R\right)  $ and $D\left(  R^{T}\right)  $ in \cite{BruSal22}, and have been
called $\Delta_{0}$ and $\Delta_{1}$ in \cite[Theorem 10.9]{Bjorne95}.
(Some of these sources are slightly imprecise around the empty set, or use
a different notion of simplicial complexes that defines them to consist of
nonempty subsets. In substance, all definitions agree when the sets $X$
and $Y$ are nonempty; we believe that ours is best suited for the empty
case.)

% \mb{Classically, a simplicial complex consists of nonempty sets, so I do not
% think the formulations in \cite{Dowker52} and \cite{Bjorne95} are incorrect. At
% least Bjorner also uses this definition.}
% (However, all these sources give slightly wrong definitions, as they neglect
% to declare $\varnothing$ to be $X$-conic and $Y$-conic. This does not change
% anything unless one of the sets $X$ and $Y$ is empty, in which case it renders
% Theorem \ref{thm.dowker} invalid.)

\begin{example}
\label{exa.dowker.2}Let $X$, $Y$ and $R$ be as in Example \ref{exa.dowker.1}.
Then, it is easy to see that%
\[
C_{X}=\left\{  \text{all }U\subseteq X\text{ such that }\left\{  3,4\right\}
\not \subseteq U\right\}  \ \ \ \ \ \ \ \ \ \ \text{and}%
\ \ \ \ \ \ \ \ \ \ C_{Y}=\left\{  \text{all }V\subseteq Y\right\}  .
\]
(The latter is because $1$ is an $X$-neighbor of any subset $V\subseteq Y$.)
\end{example}

\begin{remark}
\label{rmk.bipar} The binary relation $R$ can be re-encoded as a bipartite
graph whose black vertices are the elements of $X$, whose white vertices are
the elements of $Y$, and whose edges are the edges $\left\{  x,y\right\}  $
for all $x\in X$ and $y\in Y$ satisfying $x\ R\ y$. Thus, a $Y$-conic set
means a set of black vertices that is either empty or has a common neighbor,
whereas an $X$-conic set means a set of white vertices that is either empty or
has a common neighbor.
\end{remark}

Dowker's theorem (in Bj\"{o}rner's version \cite[Theorem 10.9]{Bjorne95}) says
the following:

\begin{theorem}
[Dowker's theorem]\label{thm.dowker}The complexes $C_{X}$ and $C_{Y}$ are homotopy-equivalent.
\end{theorem}

In \cite[Theorem 4.4]{Barmak11}, Barmak strengthened this theorem to a
simple-homotopy-equivalence:

\begin{theorem}
\label{thm.barmak}The complexes $C_{X}$ and $C_{Y}$ are
simple-homotopy-equivalent. %(i.e., have the same simple homotopy type).
\end{theorem}

In the two following sections, we shall reprove Theorem \ref{thm.barmak} (and
thus also Theorem \ref{thm.dowker}) using discrete Morse theory, by exhibiting
a larger complex that collapses to each of $C_{X}$ and $C_{Y}$. (By
\cite[Proposition 9.28]{Kozlov20}, this suffices to show that $C_{X}$ and
$C_{Y}$ are simple-homotopy-equivalent.) We note that this claim is purely
combinatorial, requiring no topology to state, and our proof will be similarly combinatorial.

\section{Reminders on discrete Morse theory}

Before we step to the actual proof, let us quickly recall some preliminaries.

First, we recall the notion of a \emph{subcomplex}:

\begin{definition}
Let $\Delta$ and $\Gamma$ be two complexes (possibly with different ground
sets). Then, we say that $\Gamma$ is a \emph{subcomplex} of $\Delta$ if and
only if $\Gamma\subseteq\Delta$.
\end{definition}

Next, and more importantly, let us recall the discrete Morse theory that we
will need. In this, we follow \cite[Chapter 10]{Kozlov20}, adapted slightly to
simplify our life (copying several definitions from \cite[\S 7]{GrKaLe21}). We
begin with the definition of
an elementary simplicial collapse %and an elementary simplicial expansion
(\cite[\S 9.1]{Kozlov20}):

\begin{definition}
  \label{def.elementary.collapse}
  Let $\Gamma$ be a subcomplex of a simplicial complex $\Delta$. If there 
  exist faces $\tau$ and $\sigma$ in $\Delta \setminus \Gamma$
  such that
  \[
  \Delta = \Gamma \cup \set{\tau, \sigma}
  \qquad \text{ and } \qquad \tau \prec \sigma ,
  \]
  then we say that $\Delta$ \emph{collapses to} $\Gamma$
  by an \emph{elementary simplicial collapse}.
  (Equivalently, we say in this case that
  the inclusion $\Gamma \subseteq \Delta$ is an
  \emph{elementary simplicial expansion}.)
\end{definition}

This definition is equivalent to \cite[Definition 9.1]{Kozlov20}.
(Indeed, \cite[Definition 9.1]{Kozlov20} requires that the
only faces of $\Delta$ that contain $\tau$ are $\tau$ and $\sigma$,
whereas we instead require that $\Gamma$ is a subcomplex;
but these two requirements imply each other.)

Elementary simplicial collapses can be composed, leading to
the notion of (non-elementary) collapses
(\cite[Definition 9.4]{Kozlov20}):

\begin{definition}
  \label{def.simple.homotopy.collapse}
  Let $\Gamma$ and $\Delta$ be two simplicial complexes.
  We say that $\Delta$ \emph{collapses to} $\Gamma$
  if $\Gamma$ can be obtained from $\Delta$ by a sequence
  (possibly empty) of elementary simplicial collapses
  -- i.e., if there exist finitely many simplicial complexes $\Delta_0, \Delta_1, \dots, \Delta_n$
  with $\Delta_0 = \Delta$ and $\Delta_n = \Gamma$ and such that
  for each $i \in \set{0, 1, \dots, n-1}$,
  the complex $\Delta_i$ collapses to $\Delta_{i+1}$ by an
  elementary collapse.
  (This clearly implies that $\Gamma$ is a subcomplex of
  $\Delta$.)
  
  We shall use the notation ``$\Delta\searrow\Gamma$''\ for
  the statement ``$\Delta$ collapses to
  $\Gamma$''.
\end{definition}

Composing elementary simplicial collapses with
their inverses leads to the weaker notion of
simple-homotopy-equivalence
(\cite[Definition 9.26]{Kozlov20}, where it is called
``having the same simple homotopy type''):

\begin{definition}
  \label{def.simple.homotopy.equivalence}
  % Let $H$ be the homotopy category of the category of simplicial complexes.
  % The objects of $H$ are all simplicial complexes. A morphism $K \to L$ in $H$ consists of 
  % an equivalence class of 
  % maps of the form $|K| \to |L|$ between the geometric realizations of the simplicial complexes $K$ and $L$, 
  % where two maps are equivalent if they are homotopic.
  % The category of simple homotopy equivalences of simplicial complexes is the smallest subcategory of $H$ 
  % containing
  % % all homeomorphisms of geometric realizations of simplicial complexes and also the 
  % the
  % geometric realization of every elementary simplicial expansion.

  Two simplicial complexes $\Gamma$ and $\Delta$ are said to be \emph{simple-homotopy-equivalent} if
  there exist finitely many simplicial complexes $\Delta_0, \Delta_1, \dots, \Delta_n$
  with $\Delta_0 = \Delta$ and $\Delta_n = \Gamma$ and with the following property:
  For each $i \in \set{0, 1, \dots, n-1}$,
  either $\Delta_{i+1}$ collapses to $\Delta_i$ or $\Delta_i$ collapses to $\Delta_{i+1}$.
\end{definition}

Clearly, simple-homotopy-equivalence is an equivalence relation.

Next, we recall the definition of a matching (\cite[Definition 10.6]{Kozlov20},
specialized to subposets of a complex):

\begin{definition}
\label{def.parmat} Let $\Delta$ be a simplicial complex. A \emph{partial
matching} (or \emph{matching} for short) on $\Delta$ shall mean a pair
$\left(  M,\mu\right)  $, where $M$ is a subset of $\Delta$ (that is, a set of
faces of $\Delta$), and where $\mu:M\rightarrow M$ is an involution (that is,
a map satisfying $\mu\circ\mu=\operatorname*{id}$) with the property that each
$F\in M$ satisfies%
\[
\text{either }\mu\left(  F\right)  \prec F\text{ or }\mu\left(  F\right)
\succ F.
\]

Note that $M$ is uniquely determined by $\mu$ (namely, as the domain of $\mu
$), so that we will refer to $\mu$ alone as a matching.
\end{definition}

\begin{example}
\label{exa.parmat.0} Let $W=\left\{  1,2,3\right\}  $. Let $\Delta$ be the
complex with ground set $W$ that contains all the eight subsets of $W$ as
faces. Set $M:=\Delta\setminus\left\{  \varnothing,W\right\}  $, and define a
map $\mu:M\rightarrow M$ by%
\begin{align*}
\mu\left(  \left\{  1\right\}  \right)   &  =\left\{  1,2\right\}  ,\qquad
\mu\left(  \left\{  2\right\}  \right)  =\left\{  2,3\right\}  ,\qquad
\mu\left(  \left\{  3\right\}  \right)  =\left\{  3,1\right\}  ,\\
\mu\left(  \left\{  1,2\right\}  \right)   &  =\left\{  1\right\}  ,\qquad
\mu\left(  \left\{  2,3\right\}  \right)  =\left\{  2\right\}  ,\qquad
\mu\left(  \left\{  3,1\right\}  \right)  =\left\{  3\right\}  .
\end{align*}
Then, $\left(  M,\mu\right)  $ is a matching on $\Delta$.
\end{example}

Partial matchings on simplicial complexes are useful for counting purposes
(being a simple kind of sign-reversing involutions; cf. \cite[Chapter
2]{Sagan19}). However, under a certain condition, they become powerful tools
for understanding the topology of $\Delta$. This condition is known as
\textquotedblleft acyclicity\textquotedblright, and is defined as follows
(\cite[Definition 10.7]{Kozlov20}):

\begin{definition}
Let $\Delta$ be a simplicial complex. Let $\left(  M,\mu\right)  $ be a
matching on $\Delta$.

\begin{enumerate}
\item[\textbf{(a)}] A \emph{cycle} of $\mu$ means an $n$-tuple $\left(
F_{1},F_{2},\ldots,F_{n}\right)  $ of distinct faces in $M$ such that $n\geq2$
and%
\[
F_{1}\succ\mu\left(  F_{1}\right)  \prec F_{2}\succ\mu\left(  F_{2}\right)
\prec F_{3}\succ\cdots\prec F_{n}\succ\mu\left(  F_{n}\right)  \prec F_{1}%
\]
(that is, such that each $i\in\left\{  1,2,\ldots,n\right\}  $ satisfies
$F_{i}\succ\mu\left(  F_{i}\right)  \prec F_{i+1}$, where $F_{n+1}:=F_{1}$).

\item[\textbf{(b)}] The matching $\mu$ is said to be \emph{acyclic} if it has
no cycle.
\end{enumerate}
\end{definition}

\begin{example}
Let $W$, $\Delta$, $M$ and $\mu$ be as in Example \ref{exa.parmat.0}. Then,
the $3$-tuple $\left(  \left\{  1,2\right\}  ,\ \left\{  3,1\right\}
,\ \left\{  2,3\right\}  \right)  $ is a cycle of the matching $\mu$, since%
\[
\left\{  1,2\right\}  \succ\underbrace{\left\{  1\right\}  }_{=\mu\left(
\left\{  1,2\right\}  \right)  }\prec\left\{  3,1\right\}  \succ
\underbrace{\left\{  3\right\}  }_{=\mu\left(  \left\{  3,1\right\}  \right)
}\prec\left\{  2,3\right\}  \succ\underbrace{\left\{  2\right\}  }%
_{=\mu\left(  \left\{  2,3\right\}  \right)  }\prec\left\{  1,2\right\}  .
\]
Thus, the matching $\mu$ is not acyclic. However, if we remove the faces
$\left\{  1\right\}  $ and $\left\{  1,2\right\}  $ from $M$ (and
correspondingly restrict $\mu$ to the set of the remaining four faces in $M$),
then the cycle disappears, and we obtain an acyclic matching.
\end{example}

Discrete Morse theory can be seen as the study of acyclic matchings. Roughly
speaking, an acyclic matching $\left(  M,\mu\right)  $ on a complex $\Delta$
allows us to \textquotedblleft cancel\textquotedblright\ the faces in $M$ when
computing homotopical or homological data, albeit the precise meaning of
\textquotedblleft cancelling\textquotedblright\ here depends on the situation.
Different instances of this principle can be found in \cite[Chapters 10, 11,
12, 13]{Kozlov20}; we will only use the following one (a part of \cite[Theorem
10.9]{Kozlov20}):

\begin{theorem}
\label{thm.dmt.collapse}Let $\Gamma$ be a subcomplex of a simplicial complex
$\Delta$. Assume that there exists an acyclic matching $\left(  M,\mu\right)
$ on $\Delta$ with $M=\Delta\setminus\Gamma$ (that is, $M$ consists of all
faces of $\Delta$ that don't belong to $\Gamma$). Then, $\Delta\searrow\Gamma$.
\end{theorem}

Acyclic matchings are combinatorial objects and can often be constructed by
hand. However, there is also a number of \textquotedblleft
macros\textquotedblright\ available to construct them from simpler objects,
such as the pairing lemma of Linusson and Shareshian (\cite[Lemma
3.4]{LinSha03}). We will use one such \textquotedblleft
macro\textquotedblright\ -- actually a version of \cite[Lemma 3.4]{LinSha03}
adapted to our situation:

\begin{lemma}
\label{lem.pairing2}Let $W$ be a finite partially ordered set. Let $\Delta$ be
a simplicial complex with ground set $W$. Let $M$ be a subset of $\Delta$. Let
$f:M\rightarrow W$ be a function. Consider the following two conditions:

\begin{enumerate}
\item[\textbf{(C1)}] For any $F\in M$, we have $F\cup\left\{  f\left(
F\right)  \right\}  \in M$ and $F\setminus\left\{  f\left(  F\right)
\right\}  \in M$ and%
\begin{equation}
f\left(  F\cup\left\{  f\left(  F\right)  \right\}  \right)  =f\left(
F\setminus\left\{  f\left(  F\right)  \right\}  \right)  =f\left(  F\right)  .
\label{eq.lem.pairing2.A2}%
\end{equation}

\item[\textbf{(C2)}] For any $F\in M$ and any $G\in M$ satisfying $G\subseteq
F$, we have $f\left(  F\right)  \leq f\left(  G\right)  $.
\end{enumerate}

Assume that Condition \textbf{(C1)} holds. Then:

\begin{enumerate}
\item[\textbf{(a)}] The map
\begin{align*}
\mu:M  &  \rightarrow M,\\
F  &  \mapsto%
\begin{cases}
F\setminus\left\{  f\left(  F\right)  \right\}  , & \text{if }f\left(
F\right)  \in F;\\
F\cup\left\{  f\left(  F\right)  \right\}  , & \text{if }f\left(  F\right)
\notin F
\end{cases}
\end{align*}
is well-defined, and the pair $\left( M, \mu\right) $ is a matching on
$\Delta$.

\item[\textbf{(b)}] If Condition \textbf{(C2)} holds as well, then this
matching $\left(  M,\mu\right)  $ is acyclic.
\end{enumerate}
\end{lemma}

\begin{proof}
Condition \textbf{(C1)} ensures that each $F\in M$ satisfies $F\setminus
\left\{  f\left(  F\right)  \right\}  \in M$ and $F\cup\left\{  f\left(
F\right)  \right\}  \in M$ and therefore
\[%
\begin{cases}
F\setminus\left\{  f\left(  F\right)  \right\}  , & \text{if }f\left(
F\right)  \in F;\\
F\cup\left\{  f\left(  F\right)  \right\}  , & \text{if }f\left(  F\right)
\notin F
\end{cases}
\in M.
\]
Hence, the map
\begin{align*}
\mu:M  &  \rightarrow M,\\
F  &  \mapsto%
\begin{cases}
F\setminus\left\{  f\left(  F\right)  \right\}  , & \text{if }f\left(
F\right)  \in F;\\
F\cup\left\{  f\left(  F\right)  \right\}  , & \text{if }f\left(  F\right)
\notin F
\end{cases}
\end{align*}
is well-defined. We shall now show that this map is an involution:

\begin{statement}
\textit{Claim 1:} We have $\mu\left(  \mu\left(  G\right)  \right)  =G$ for
each $G\in M$.
\end{statement}

\begin{proof}
[Proof of Claim 1.]Let $G\in M$. We must show that $\mu\left(  \mu\left(
G\right)  \right)  =G$. We are in one of the following two cases:

\textit{Case 1:} We have $f\left(  G\right)  \in G$.

\textit{Case 2:} We have $f\left(  G\right)  \notin G$.

Let us first consider Case 1. In this case, we have $f\left(  G\right)  \in
G$. Hence, the definition of $\mu$ yields $\mu\left(  G\right)  =G\setminus
\left\{  f\left(  G\right)  \right\}  $. Let $F=\mu\left(  G\right)  $. Thus,
$F=\mu\left(  G\right)  =G\setminus\left\{  f\left(  G\right)  \right\}  $.
Hence, $f\left(  G\right)  \notin F$.

\begin{vershort}
The second equality sign in the equality (\ref{eq.lem.pairing2.A2}) (applied
to $G$ instead of $F$) yields $f\left(  G\setminus\left\{  f\left(  G\right)
\right\}  \right)  =f\left(  G\right)  $. In other words, $f\left(  F\right)
=f\left(  G\right)  $ (since $F=G\setminus\left\{  f\left(  G\right)
\right\}  $). Thus, $f\left(  F\right)  =f\left(  G\right)  \notin F$. Hence,
the definition of $\mu$ yields
\[
\mu\left(  F\right)  =\underbrace{F}_{=G\setminus\left\{  f\left(  G\right)
\right\}  }\cup\left\{  \underbrace{f\left(  F\right)  }_{=f\left(  G\right)
}\right\}  =\left(  G\setminus\left\{  f\left(  G\right)  \right\}  \right)
\cup\left\{  f\left(  G\right)  \right\}  =G
\]
(since $f\left(  G\right)  \in G$). In other words, $\mu\left(  \mu\left(
G\right)  \right)  =G$ (since $F=\mu\left(  G\right)  $). This proves Claim 1
in Case 1.
\end{vershort}

\begin{verlong}
The equality (\ref{eq.lem.pairing2.A2}) (applied to $G$ instead of $F$) yields
$f\left(  G\cup\left\{  f\left(  G\right)  \right\}  \right)  =f\left(
G\setminus\left\{  f\left(  G\right)  \right\}  \right)  =f\left(  G\right)
$. Thus, in particular, $f\left(  G\setminus\left\{  f\left(  G\right)
\right\}  \right)  =f\left(  G\right)  $. In other words, $f\left(  F\right)
=f\left(  G\right)  $ (since $F=G\setminus\left\{  f\left(  G\right)
\right\}  $). Thus, $f\left(  F\right)  =f\left(  G\right)  \notin F$. Hence,
the definition of $\mu$ yields
\[
\mu\left(  F\right)  =\underbrace{F}_{=G\setminus\left\{  f\left(  G\right)
\right\}  }\cup\left\{  \underbrace{f\left(  F\right)  }_{=f\left(  G\right)
}\right\}  =\left(  G\setminus\left\{  f\left(  G\right)  \right\}  \right)
\cup\left\{  f\left(  G\right)  \right\}  =G
\]
(since $f\left(  G\right)  \in G$). In other words, $\mu\left(  \mu\left(
G\right)  \right)  =G$ (since $F=\mu\left(  G\right)  $). This proves Claim 1
in Case 1.
\end{verlong}

Let us now consider Case 2. In this case, we have $f\left(  G\right)  \notin
G$. Hence, the definition of $\mu$ yields $\mu\left(  G\right)  =G\cup\left\{
f\left(  G\right)  \right\}  $. Let $F=\mu\left(  G\right)  $. Thus,
$F=\mu\left(  G\right)  =G\cup\left\{  f\left(  G\right)  \right\}  $. Hence,
$f\left(  G\right)  \in F$.

\begin{vershort}
From (\ref{eq.lem.pairing2.A2}) (applied to $G$ instead of $F$), we obtain
$f\left(  G\cup\left\{  f\left(  G\right)  \right\}  \right)  =f\left(
G\right)  $. In other words, $f\left(  F\right)  =f\left(  G\right)  $ (since
$F=G\cup\left\{  f\left(  G\right)  \right\}  $). Thus, $f\left(  F\right)
=f\left(  G\right)  \in F$. Hence, the definition of $\mu$ yields
\[
\mu\left(  F\right)  =\underbrace{F}_{=G\cup\left\{  f\left(  G\right)
\right\}  }\setminus\left\{  \underbrace{f\left(  F\right)  }_{=f\left(
G\right)  }\right\}  =\left(  G\cup\left\{  f\left(  G\right)  \right\}
\right)  \setminus\left\{  f\left(  G\right)  \right\}  =G
\]
(since $f\left(  G\right)  \notin G$). In other words, $\mu\left(  \mu\left(
G\right)  \right)  =G$ (since $F=\mu\left(  G\right)  $). This proves Claim 1
in Case 2.
\end{vershort}

\begin{verlong}
The equality (\ref{eq.lem.pairing2.A2}) (applied to $G$ instead of $F$) yields
$f\left(  G\cup\left\{  f\left(  G\right)  \right\}  \right)  =f\left(
G\setminus\left\{  f\left(  G\right)  \right\}  \right)  =f\left(  G\right)
$. Thus, in particular, $f\left(  G\cup\left\{  f\left(  G\right)  \right\}
\right)  =f\left(  G\right)  $. In other words, $f\left(  F\right)  =f\left(
G\right)  $ (since $F=G\cup\left\{  f\left(  G\right)  \right\}  $). Thus,
$f\left(  F\right)  =f\left(  G\right)  \in F$. Hence, the definition of $\mu$
yields
\[
\mu\left(  F\right)  =\underbrace{F}_{=G\cup\left\{  f\left(  G\right)
\right\}  }\setminus\left\{  \underbrace{f\left(  F\right)  }_{=f\left(
G\right)  }\right\}  =\left(  G\cup\left\{  f\left(  G\right)  \right\}
\right)  \setminus\left\{  f\left(  G\right)  \right\}  =G
\]
(since $f\left(  G\right)  \notin G$). In other words, $\mu\left(  \mu\left(
G\right)  \right)  =G$ (since $F=\mu\left(  G\right)  $). This proves Claim 1
in Case 2.
\end{verlong}

We have now proved Claim 1 in both Cases 1 and 2. Hence, Claim 1 is proved.
\end{proof}

As a byeffect of our proof of Claim 1, we obtain the following:

\begin{statement}
\textit{Claim 2:} Let $G\in M$. Then, $f\left(  \mu\left(  G\right)  \right)
=f\left(  G\right)  $.
\end{statement}

\begin{proof}
[Proof of Claim 2.]Let $F=\mu\left(  G\right)  $. Then, $f\left(  F\right)
=f\left(  G\right)  $ (indeed, this equality has been proved in our above
proof of Claim 1, in both Cases 1 and 2). In other words, $f\left(  \mu\left(
G\right)  \right)  =f\left(  G\right)  $ (since $F=\mu\left(  G\right)  $).
This proves Claim 2.
\end{proof}

Here is a simple set-theoretical fact that will come useful later:

\begin{statement}
\textit{Claim 3:} Let $F,G,H$ be three subsets of $W$ such that $F\succ H\prec
G$. Let $s$ be an element such that $s\in F$ and $s\in G$ but $s\notin H$.
Then, $F=G$.
\end{statement}

\begin{vershort}

\begin{proof}
[Proof of Claim 3.]Left to the reader. (Show that both $F$ and $G$ equal
$H\cup\left\{  s\right\}  $.)
\end{proof}
\end{vershort}

\begin{verlong}

\begin{proof}
[Proof of Claim 3.]We have $H\prec F$ (since $F\succ H$). In other words,
$H\subseteq F$ and $\left\vert F\setminus H\right\vert =1$. In other words,
$F=H\cup\left\{  t\right\}  $ for some element $t\in F$. Consider this $t$.
Since $s\in F=H\cup\left\{  t\right\}  $, we must have either $s\in H$ or
$s\in\left\{  t\right\}  $. Since $s\notin H$, we thus have $s\in\left\{
t\right\}  $, so that $s=t$. Hence, $t=s$, so that $F=H\cup\left\{
\underbrace{t}_{=s}\right\}  =H\cup\left\{  s\right\}  $.

The same argument (applied to $G$ instead of $F$) shows that $G=H\cup\left\{
s\right\}  $ (since $H\prec G$). Comparing these two equalities, we obtain
$F=G$. This proves Claim 3.
\end{proof}
\end{verlong}

Now, we prove one more claim about $\mu$:

\begin{statement}
\textit{Claim 4:} Let $F,G\in M$ be such that $F\succ\mu\left(  F\right)
\prec G\succ\mu\left(  G\right)  $ and $f\left(  F\right)  =f\left(  G\right)
$. Then, $F=G$.
\end{statement}

\begin{proof}
[Proof of Claim 4.]If we had $f\left(  F\right)  \notin F$, then the
definition of $\mu$ would yield $\mu\left(  F\right)  =F\cup\left\{  f\left(
F\right)  \right\}  \supseteq F$, which would contradict $F\succ\mu\left(
F\right)  $. Thus, we must have $f\left(  F\right)  \in F$. Similarly,
$f\left(  G\right)  \in G$ (since $G\succ\mu\left(  G\right)  $).

Let us denote the element $f\left(  F\right)  =f\left(  G\right)  $ by $s$.
Thus, $s=f\left(  F\right)  \in F$ and $s=f\left(  G\right)  \in G$.

We have $f\left(  F\right)  \in F$. Thus, the definition of $\mu$ yields
$\mu\left(  F\right)  =F\setminus\left\{  f\left(  F\right)  \right\}
=F\setminus\left\{  s\right\}  $ (since $f\left(  F\right)  =s$). Hence,
$s\notin\mu\left(  F\right)  $ (since $s\notin F\setminus\left\{  s\right\}
$). Thus, Claim 3 (applied to $H=\mu\left(  F\right)  $) yields $F=G$. This
proves Claim 4.
\end{proof}

Claim 1 shows that $\mu\circ\mu=\operatorname*{id}$. In other words, the map
$\mu$ is an involution. Moreover, it has the property that each $F\in M$
satisfies%
\[
\text{either }\mu\left(  F\right)  \prec F\text{ or }\mu\left(  F\right)
\succ F
\]
(since we either have $f\left(  F\right)  \in F$, in which case $\mu\left(
F\right)  =F\setminus\left\{  f\left(  F\right)  \right\}  \prec F$, or we
have $f\left(  F\right)  \notin F$, in which case $\mu\left(  F\right)
=F\cup\left\{  f\left(  F\right)  \right\}  \succ F$). Thus, the map $\mu$
(or, to be more precise, the pair $\left(  M,\mu\right)  $) is a matching on
$\Delta$. This proves Lemma \ref{lem.pairing2} \textbf{(a)}. \medskip

\textbf{(b)} Assume that Condition \textbf{(C2)} holds as well. Lemma
\ref{lem.pairing2} \textbf{(a)} shows that the map $\mu$ (defined in said
lemma) is well-defined, and that the pair $\left( M, \mu\right) $ is a
matching on $\Delta$.

It remains to show that this matching $\left(  M,\mu\right)  $ is acyclic. In
other words, we must show that it has no cycle.

Assume the contrary. Thus, $\left(  M,\mu\right)  $ has a cycle. By the
definition of a cycle, this cycle is an $n$-tuple $\left(  F_{1},F_{2}%
,\ldots,F_{n}\right)  $ of distinct faces in $M$ such that $n\geq2$ and%
\[
F_{1}\succ\mu\left(  F_{1}\right)  \prec F_{2}\succ\mu\left(  F_{2}\right)
\prec F_{3}\succ\cdots\prec F_{n}\succ\mu\left(  F_{n}\right)  \prec F_{1}.
\]
Consider this $n$-tuple $\left(  F_{1},F_{2},\ldots,F_{n}\right)  $. Set
$F_{n+1}:=F_{1}$. Thus, each $i\in\left\{  1,2,\ldots,n\right\}  $ satisfies
\begin{equation}
F_{i}\succ\mu\left(  F_{i}\right)  \prec F_{i+1}.
\label{pf.lem.pairing2.zigzag}%
\end{equation}
Recall that the $n$ faces $F_{1},F_{2},\ldots,F_{n}$ are distinct; thus,
$F_{1}\neq F_{2}$ (since $n\geq2$).

Now, let $i\in\left\{  1,2,\ldots,n\right\}  $. Then, Claim 2 (applied to
$G=F_{i}$) yields $f\left(  \mu\left(  F_{i}\right)  \right)  =f\left(
F_{i}\right)  $. Furthermore, $\mu\left(  F_{i}\right)  \subseteq F_{i+1}$
(since (\ref{pf.lem.pairing2.zigzag}) yields $\mu\left(  F_{i}\right)  \prec
F_{i+1}$). Thus, Condition \textbf{(C2)} (applied to $F=F_{i+1}$ and
$G=\mu\left(  F_{i}\right)  $) yields $f\left(  F_{i+1}\right)  \leq f\left(
\mu\left(  F_{i}\right)  \right)  =f\left(  F_{i}\right)  $. In other words,
$f\left(  F_{i}\right)  \geq f\left(  F_{i+1}\right)  $.

Forget that we fixed $i$. We thus have proved that $f\left(  F_{i}\right)
\geq f\left(  F_{i+1}\right)  $ for each $i\in\left\{  1,2,\ldots,n\right\}
$. In other words,%
\begin{equation}
f\left(  F_{1}\right)  \geq f\left(  F_{2}\right)  \geq\cdots\geq f\left(
F_{n+1}\right)  . \label{pf.lem.pairing2.chain}%
\end{equation}
Hence, in particular, $f\left(  F_{1}\right)  \geq f\left(  F_{2}\right)  $
and $f\left(  F_{2}\right)  \geq f\left(  F_{n+1}\right)  $. Hence, $f\left(
F_{2}\right)  \geq f\left(  F_{n+1}\right)  =f\left(  F_{1}\right)  $ (since
$F_{n+1}=F_{1}$). Combining this with $f\left(  F_{1}\right)  \geq f\left(
F_{2}\right)  $, we obtain $f\left(  F_{1}\right)  =f\left(  F_{2}\right)  $.

Thus, Claim 4 (applied to $F=F_{1}$ and $G=F_{2}$) yields $F_{1}=F_{2}$ (since
$F_{1}\succ\mu\left(  F_{1}\right)  \prec F_{2}\succ\mu\left(  F_{2}\right)
$). But this contradicts $F_{1}\neq F_{2}$. This contradiction shows that our
assumption was wrong. Hence, we have shown that the matching $\left(
M,\mu\right)  $ has no cycle, i.e., is acyclic. This proves Lemma
\ref{lem.pairing2} \textbf{(b)}.
\end{proof}

\begin{noncompile}
OLD lemmas:

To prove this theorem, we will use the \emph{pairing lemma} of Linusson and
Shareshian (\cite[Lemma 3.4]{LinSha03}):

\begin{lemma}
\label{lem.pairing}Let $V$ be a partially ordered set. Let $\Sigma$ be a
simplicial complex with ground set $V$. Let $P$ be a subset of $\Sigma$. Let
$f:P\rightarrow V$ be a function. Set%
\begin{align*}
P_{-}  &  :=\left\{  F\in P\ \mid\ f\left(  F\right)  \notin F\right\}
\ \ \ \ \ \ \ \ \ \ \text{and}\\
P_{+}  &  :=\left\{  F\in P\ \mid\ f\left(  F\right)  \in F\right\}  .
\end{align*}
For any $F\in P_{-}$, define
\[
F^{+}:=F\cup\left\{  f\left(  F\right)  \right\}  .
\]
For any $G\in P_{+}$, define%
\[
G^{-}:=G\setminus\left\{  f\left(  G\right)  \right\}  .
\]
Assume that the following conditions are satisfied:

\begin{enumerate}
\item[\textbf{(A)}] For any $F\in P_{-}$, we have $F^{+}\in P$.

\item[\textbf{(B)}] For any $F\in P_{-}$, we have $f\left(  F^{+}\right)
=f\left(  F\right)  $.

\item[\textbf{(C)}] For any $G\in P_{+}$ satisfying $G^{-}\in P$, we have
$f\left(  G^{-}\right)  =f\left(  G\right)  $.

\item[\textbf{(D)}] For any $F\in P_{-}$ and any $x\in F$ satisfying
$F^{+}\setminus\left\{  x\right\}  \in P$, we have $f\left(  F\right)  \leq
f\left(  F^{+}\setminus\left\{  x\right\}  \right)  $.
\end{enumerate}

Then, the set $\left\{  \left\{  F^{+},F\right\}  \ \mid\ F\in P_{-}\right\}
$ is an acyclic matching on $P$, and the critical simplices of this matching
are those $G\in P_{+}$ that satisfy $G^{-}\notin P$.
\end{lemma}

Note that the sets $P_{-}$ and $P_{+}$ are denoted by $P_{f}$ and $P\setminus
P_{f}$ in \cite[Lemma 3.4]{LinSha03}, but we find our notations more
intuitive. Finally, the notation for matchings in \cite[Lemma 3.4]{LinSha03}
differs from ours, in that it uses directed edges $\left(  F^{+},F\right)  $
instead of undirected edges $\left\{  F^{+},F\right\}  $. Finally, Linusson
and Shareshian, in \cite[Lemma 3.4]{LinSha03}, additionally require the subset
$P$ of $\Sigma$ to be order-convex, but this requirement is not necessary for
their proof.

Lemma \ref{lem.pairing} has the following simple consequence:

\begin{lemma}
\label{lem.pairing.complete}Let $V$, $\Sigma$, $P$, $f$, $P_{-}$, $P_{+}$,
$F^{+}$ and $G^{-}$ be as in Lemma \ref{lem.pairing}. Assume that the
following additional condition is satisfied:

\begin{enumerate}
\item[\textbf{(E)}] For any $G\in P_{+}$, we have $G^{-}\in P$.
\end{enumerate}

Then, the set $\left\{  \left\{  F^{+},F\right\}  \ \mid\ F\in P_{-}\right\}
$ is a complete acyclic matching on $P$ (that is, an acyclic matching on $P$
that has no critical simplices).
\end{lemma}
\end{noncompile}

\section{The biclique complex of a bipartite relation}

Let $R$ be relation from a set $X$ to a set $Y$. We say that $R$ is a
\emph{bipartite relation} if the sets $X$ and $Y$ are disjoint.
In this section we assume that $R$ is a bipartite relation.
% We assume that the sets $X$ and $Y$ are disjoint. (Later we shall lift this
% assumption in proving Theorem \ref{thm.dowker} and Theorem \ref{thm.barmak}.)

A \emph{biclique} of the relation $R$ shall mean a set of the form $U\cup V$,
where $U\subseteq X$ and $V\subseteq Y$ are two nonempty subsets satisfying
$U\ \mathbf{R}\ V$.
(Note the \textquotedblleft nonempty\textquotedblright\ requirement! Thus, a
biclique must intersect both $X$ and $Y$.)
In this section, we will only deal with bicliques of $R$, so we will just
refer to them as ``bicliques''.

Clearly, any biclique is a subset of $X\cup Y$.
The following facts are near-trivial:

\begin{lemma}
\label{lem.XuY1}Let $U\subseteq X$ and $V\subseteq Y$ be two subsets. Let
$F=U\cup V$. Then, $F\cap X=U$ and $F\cap Y=V$.
\end{lemma}

\begin{vershort}

\begin{proof}
This is because $X$ and $Y$ are disjoint.
\end{proof}
\end{vershort}

\begin{verlong}

\begin{proof}
From $V\subseteq Y$, we obtain $V\cap X\subseteq Y\cap X=X\cap Y=\varnothing$
(since $X$ and $Y$ are disjoint). Thus, $V\cap X=\varnothing$. Furthermore,
from $U\subseteq X$, we obtain $U\cap X=U$.

From $F=U\cup V$, we obtain%
\[
F\cap X=\left(  U\cup V\right)  \cap X=\underbrace{\left(  U\cap X\right)
}_{=U}\cup\underbrace{\left(  V\cap X\right)  }_{=\varnothing}=U\cup
\varnothing=U.
\]
Similarly, $F\cap Y=V$. Thus, Lemma \ref{lem.XuY1} is proved.
\end{proof}
\end{verlong}

\begin{lemma}
\label{lem.biclique.XY}Let $G$ be a subset of a biclique. Then:

\begin{enumerate}
\item[\textbf{(a)}] If $G\subseteq X$, then $G\in C_{X}$.

\item[\textbf{(b)}] If $G\subseteq Y$, then $G\in C_{Y}$.
\end{enumerate}
\end{lemma}

\begin{proof}
\textbf{(a)} Assume that $G\subseteq X$. We know that $G$ is a subset of a
biclique. Let $F$ be this biclique. Then, $G\subseteq F$.

But $F$ is a biclique. In other words, $F=U\cup V$ for some nonempty subsets
$U\subseteq X$ and $V\subseteq Y$ satisfying $U\ \mathbf{R}\ V$ (by the
definition of a biclique). Consider these $U$ and $V$.

\begin{vershort}
Combining $G\subseteq F$ and $G\subseteq X$, we obtain $G\subseteq F\cap X=U$
(by Lemma \ref{lem.XuY1}).

There exists some $v\in V$ (since $V$ is nonempty). This $v$ is then a
$Y$-neighbor of $U$ (since $U\ \mathbf{R}\ V$), and thus a $Y$-neighbor of $G$
(since $G\subseteq U$). Thus, the set $G$ has a $Y$-neighbor, and hence is
$Y$-conic, i.e., belongs to $C_{X}$. This proves Lemma \ref{lem.biclique.XY}
\textbf{(a)}.
\end{vershort}

\begin{verlong}
From $G\subseteq X$, we obtain $G=\underbrace{G}_{\subseteq F}\cap\,X\subseteq
F\cap X=U$ (by Lemma \ref{lem.XuY1}). Hence, each $u\in G$ is a $u\in U$.

Recall that $V$ is nonempty. Thus, there exists some $v\in V$. Consider this
$v$. Then, each $u\in U$ satisfies $u\ R\ v$ (since $U\ \mathbf{R}\ V$).
Hence, each $u\in G$ satisfies $u\ R\ v$ (since each $u\in G$ is a $u\in U$).
In other words, $v$ is a $Y$-neighbor of $G$. Hence, the set $G$ has a
$Y$-neighbor (namely, $v$). Thus, the set $G$ is $Y$-conic, i.e., belongs to
$C_{X}$ (by the definition of $C_{X}$). In other words, $G\in C_{X}$. This
proves Lemma \ref{lem.biclique.XY} \textbf{(a)}.
\end{verlong}

\textbf{(b)} This is analogous to part \textbf{(a)}.
\end{proof}

\begin{noncompile}

\begin{proposition}
\label{prop.biclique.sub}Any subset of a biclique is either a biclique or a
$Y$-conic set or an $X$-conic set.
\end{proposition}
\end{noncompile}

\begin{proposition}
\label{prop.bicl-comp}Set%
\[
B:=\left\{  \text{bicliques}\right\}  \cup C_{X}\cup C_{Y}.
\]
Then, $B$ is a simplicial complex with ground set $X\cup Y$.
\end{proposition}

\begin{vershort}

\begin{proof}
Let $F\in B$. We must prove that every subset of $F$ belongs to $B$ as well.
If $F\in C_{X}$, then this is clear (since $C_{X}$ is a complex); likewise if
$F\in C_{Y}$. Thus, we only need to consider the case when $F\in\left\{
\text{bicliques}\right\}  $. So let us assume that $F\in\left\{
\text{bicliques}\right\}  $.

Thus, $F$ is a biclique. A subset of $F$ can be either fully contained in $X$,
or fully contained in $Y$, or intersect both $X$ and $Y$. In the first of
these three cases, it belongs to $C_{X}$ (by Lemma \ref{lem.biclique.XY}
\textbf{(a)}). In the second, it belongs to $C_{Y}$ (by Lemma
\ref{lem.biclique.XY} \textbf{(b)}). In the third, it is itself a biclique
(since $U\ \mathbf{R}\ V$ entails $U^{\prime}\ \mathbf{R}\ V^{\prime}$
whenever $U^{\prime}\subseteq U$ and $V^{\prime}\subseteq V$). In either case,
it thus belongs to $B$ (by the definition of $B$). Proposition
\ref{prop.bicl-comp} is thus proven.
\end{proof}
\end{vershort}

\begin{verlong}

\begin{proof}
The definition of $B$ clearly yields $\left\{  \text{bicliques}\right\}
\subseteq B$ and $C_{X}\subseteq B$ and $C_{Y}\subseteq B$. Moreover, $B$ is
clearly a set of subsets of $X\cup Y$.

We must prove that $B$ is a complex with ground set $X\cup Y$. In other words,
we must prove that $G\in B$ whenever $G$ is a subset of a set $F\in B$.

So let $F\in B$ be arbitrary, and let $G$ be a subset of $F$. We must prove
that $G\in B$.

Assume the contrary. Thus, $G\notin B$. Hence, $G\notin C_{X}$ (since $G\in
C_{X}$ would entail $G\in C_{X}\subseteq B$, contradicting $G\notin B$).
Similarly, $G\notin C_{Y}$.

But $C_{X}$ is a complex. Thus, if we had $F\in C_{X}$, then we would have
$G\in C_{X}$ as well (since $G$ is a subset of $F$), which would contradict
$G\notin C_{X}$. Hence, $F\notin C_{X}$. Similarly, $F\notin C_{Y}$.

We have $F\in B=\left\{  \text{bicliques}\right\}  \cup C_{X}\cup C_{Y}$.
Thus, we have $F\in\left\{  \text{bicliques}\right\}  $ or $F\in C_{X}$ or
$F\in C_{Y}$. Hence, we must have $F\in\left\{  \text{bicliques}\right\}  $
(since $F\notin C_{X}$ and $F\notin C_{Y}$). In other words, $F$ is a
biclique. Hence, $G$ is a subset of a biclique (since $G$ is a subset of $F$).

If we had $G\subseteq X$, then Lemma \ref{lem.biclique.XY} \textbf{(a)} would
yield $G\in C_{X}$, which would contradict $G\notin C_{X}$. Thus, we don't
have $G\subseteq X$. Similarly, we don't have $G\subseteq Y$.

The set $F$ is a biclique, and thus has the form $U\cup V$, where $U\subseteq
X$ and $V\subseteq Y$ are two nonempty subsets satisfying $U\ \mathbf{R}\ V$
(since any biclique has this form). Consider these $U$ and $V$. Thus, $F=U\cup
V$.

Set $U^{\prime}:=G\cap U$ and $V^{\prime}:=G\cap V$. Then, $U^{\prime}=G\cap
U\subseteq U\subseteq X$ and similarly $V^{\prime}\subseteq Y$. Since $G$ is a
subset of $F$, we obtain%
\[
G=G\cap\underbrace{F}_{=U\cup V}=G\cap\left(  U\cup V\right)
=\underbrace{\left(  G\cap U\right)  }_{=U^{\prime}}\cup\underbrace{\left(
G\cap V\right)  }_{=V^{\prime}}=U^{\prime}\cup V^{\prime}.
\]
If we had $U^{\prime}=\varnothing$, then we would thus obtain
$G=\underbrace{U^{\prime}}_{=\varnothing}\cup\,V^{\prime}=V^{\prime}\subseteq
Y$, which would contradict the fact that we don't have $G\subseteq Y$. Thus,
we must have $U^{\prime}\neq\varnothing$. Similarly, $V^{\prime}%
\neq\varnothing$. Hence, $U^{\prime}\subseteq X$ and $V^{\prime}\subseteq Y$
are two nonempty subsets. Moreover, $U^{\prime}\subseteq U$ and similarly
$V^{\prime}\subseteq V$. Thus, from $U\ \mathbf{R}\ V$, we obtain $U^{\prime
}\ \mathbf{R}\ V^{\prime}$ (by the definition of the relation $\mathbf{R}%
$\ \ \ \ \footnote{Here are the details of this argument: We have
$U\ \mathbf{R}\ V$. In other words, all $u\in U$ and $v\in V$ satisfy
$u\ R\ v$. Hence, all $u\in U^{\prime}$ and $v\in V^{\prime}$ satisfy
$u\ R\ v$ (since all $u\in U^{\prime}$ and $v\in V^{\prime}$ satisfy $u\in
U^{\prime}\subseteq U$ and $v\in V^{\prime}\subseteq V$, and thus (by the
preceding sentence) they satisfy $u\ R\ v$). In other words, we have
$U^{\prime}\ \mathbf{R}\ V^{\prime}$.}).

Now, we know that the set $G$ has the form $G=U^{\prime}\cup V^{\prime}$,
where $U^{\prime}\subseteq X$ and $V^{\prime}\subseteq Y$ are two nonempty
subsets satisfying $U^{\prime}\ \mathbf{R}\ V^{\prime}$. Thus, $G$ is a
biclique (by the definition of a biclique). In other words, $G\in\left\{
\text{bicliques}\right\}  \subseteq B$. This contradicts $G\notin B$. This
contradiction shows that our assumption was false, and this completes the
proof of Proposition \ref{prop.bicl-comp}.
\end{proof}
\end{verlong}

\section{Collapses of the biclique complex}

In this section, we still assume that $R$ is a bipartite relation.
Consider the simplicial complex $B$ defined in Proposition
\ref{prop.bicl-comp}.
We shall call $B$ the \emph{biclique complex} of the relation $R$.
Both $C_{X}$ and $C_{Y}$ are subcomplexes of $B$.
In this section, we show that the biclique complex $B$
collapses to each of the complexes $C_{X}$ and $C_{Y}$.
This argument will then lead us to what is perhaps the simplest
proofs of Theorem \ref{thm.dowker} and Theorem \ref{thm.barmak}.

% We will use discrete Morse theory, in the form
% presented in \cite[Chapter 10]{Kozlov20}.
Our main claim is the following:

\begin{theorem}
\label{thm.biclique.collapse}
  Let $R$ be a bipartite relation from a set $X$ to a set $Y$.
  Then the biclique complex $B$ of $R$ collapses to both $C_{X}$ and $C_{Y}$.
  That is, $B\searrow C_{X}$ and $B\searrow C_{Y}$.
\end{theorem}

\begin{proof}
[Proof of Theorem \ref{thm.biclique.collapse}.]We shall only show that
$B\searrow C_{X}$, since the proof of $B\searrow C_{Y}$ is analogous.

Let $M:=B\setminus C_{X}$. This is a subset of $B$, and consists of those
faces of $B$ that are not faces of $C_{X}$. We shall prove some easy claims:

\begin{statement}
\textit{Claim 1:} Let $F\in M$. Then, the set $F\cap Y$ is nonempty and has an
$X$-neighbor.
\end{statement}

\begin{proof}
[Proof of Claim 1.]We have $F\in M=B\setminus C_{X}$, so that $F\in B$ but
$F\notin C_{X}$. From $F\notin C_{X}$, we obtain $F\neq\varnothing$ (since the
definition of $C_{X}$ yields $\varnothing\in C_{X}$), so that $F$ is nonempty.
Moreover,
\[
F\in B=\left\{  \text{bicliques}\right\}  \cup C_{X}\cup C_{Y}.
\]
In other words, $F\in\left\{  \text{bicliques}\right\}  $ or $F\in C_{X}$ or
$F\in C_{Y}$. Since $F\notin C_{X}$, we can rule out the second possibility,
so we conclude that $F\in\left\{  \text{bicliques}\right\}  $ or $F\in C_{Y}$.
Thus, we are in one of the following two cases:

\textit{Case 1:} We have $F\in\left\{  \text{bicliques}\right\}  $.

\textit{Case 2:} We have $F\in C_{Y}$.

\begin{vershort}
Let us consider Case 1. In this case, $F$ is a biclique. In other words,
$F=U\cup V$ for some nonempty subsets $U\subseteq X$ and $V\subseteq Y$
satisfying $U\ \mathbf{R}\ V$ (by the definition of a biclique). Consider
these $U$ and $V$. Any element of $U$ is an $X$-neighbor of $V$ (since
$U\ \mathbf{R}\ V$). Hence, the set $V$ has an $X$-neighbor (since $U$ is nonempty).

So we have shown that the set $V$ is nonempty and has an $X$-neighbor. In
other words, the set $F\cap Y$ is nonempty and has an $X$-neighbor (since
Lemma \ref{lem.XuY1} yields $F\cap Y=V$). This proves Claim 1 in Case 1.
\end{vershort}

\begin{verlong}
Let us consider Case 1. In this case, we have $F\in\left\{  \text{bicliques}%
\right\}  $. In other words, $F$ is a biclique. In other words, $F=U\cup V$
for some nonempty subsets $U\subseteq X$ and $V\subseteq Y$ satisfying
$U\ \mathbf{R}\ V$ (by the definition of a biclique). Consider these $U$ and
$V$. Lemma \ref{lem.XuY1} yields $F\cap X=U$ and $F\cap Y=V$. Moreover, $U$ is
nonempty, so there exists some $x\in U$. Consider this $x$. Then, $x\ R\ v$
for each $v\in V$ (since $x\in U$ and $U\ \mathbf{R}\ V$). In other words, $x$
is an $X$-neighbor of $V$. Hence, the set $V$ has an $X$-neighbor (namely, $x$).

So we have shown that the set $V$ is nonempty and has an $X$-neighbor. In
other words, the set $F\cap Y$ is nonempty and has an $X$-neighbor (since
$F\cap Y=V$). This proves Claim 1 in Case 1.
\end{verlong}

\begin{vershort}
Let us now consider Case 2. In this case, we have $F\in C_{Y}$. In other
words, $F$ is an $X$-conic set. Since $F$ is nonempty, this entails that $F$
has an $X$-neighbor. Moreover, $F\subseteq Y$ (since $F$ is an $X$-conic set),
so that $F\cap Y=F$. Hence, the set $F\cap Y$ is nonempty and has an
$X$-neighbor (since $F$ is nonempty and has an $X$-neighbor). This proves
Claim 1 in Case 2.
\end{vershort}

\begin{verlong}
Let us now consider Case 2. In this case, we have $F\in C_{Y}$. In other
words, $F$ is an $X$-conic set (by the definition of $C_{Y}$). In other words,
$F$ is either empty or has an $X$-neighbor (by the definition of
\textquotedblleft$X$-conic\textquotedblright). Since $F$ is nonempty, this
entails that $F$ has an $X$-neighbor. Moreover, $F\subseteq Y$ (since $F$ is
an $X$-conic set), so that $F\cap Y=F$. Hence, the set $F\cap Y$ is nonempty
and has an $X$-neighbor (since $F$ is nonempty and has an $X$-neighbor). This
proves Claim 1 in Case 2.
\end{verlong}

We have now proved Claim 1 in both Cases 1 and 2. Hence, Claim 1 always holds.
\end{proof}

Let us now fix a total order on the set $X\cup Y$ (chosen arbitrarily). We
define a map $f:M\rightarrow X\cup Y$ as follows: For each face $F\in M$,
\[
\text{we let }f\left(  F\right)  \text{ be the largest }X\text{-neighbor of
the set }F\cap Y
\]
(this is well-defined, since Claim 1 shows that $F\cap Y$ has an
$X$-neighbor). Note that $f\left(  F\right)  $ actually belongs to $X$ (by the
definition of an $X$-neighbor), but we prefer to use $X\cup Y$ as the target
in order to agree with Lemma \ref{lem.pairing2} notationally.

We now claim the following:

\begin{statement}
\textit{Claim 2:} For any $F\in M$, we have $F\cup\left\{  f\left(  F\right)
\right\}  \in M$ and $F\setminus\left\{  f\left(  F\right)  \right\}  \in M$
and%
\[
f\left(  F\cup\left\{  f\left(  F\right)  \right\}  \right)  =f\left(
F\setminus\left\{  f\left(  F\right)  \right\}  \right)  =f\left(  F\right)
.
\]

\end{statement}

\begin{proof}
[Proof of Claim 2.]Let $F\in M$. Thus, $F\in M=B\setminus C_{X}$, so that
$F\in B$ and $F\notin C_{X}$.

By the definition of $f$, the element $f\left(  F\right)  $ is the largest
$X$-neighbor of the set $F\cap Y$. Thus, $f\left(  F\right)  $ is an
$X$-neighbor of $F\cap Y$, so that $f\left(  F\right)  \in X$ (by the
definition of an $X$-neighbor), and thus $f\left(  F\right)  \notin Y$ (since
the sets $X$ and $Y$ are disjoint).

Set%
\[
F^{+}:=F\cup\left\{  f\left(  F\right)  \right\}
\ \ \ \ \ \ \ \ \ \ \text{and}\ \ \ \ \ \ \ \ \ \ F^{-}:=F\setminus\left\{
f\left(  F\right)  \right\}  .
\]

\begin{vershort}
Thus, the sets $F^{+}$ and $F^{-}$ differ from $F$ only in the single element
$f\left(  F\right)  $ (if they differ from $F$ at all). Thus, they contain the
same elements of $Y$ as $F$ (since $f\left(  F\right)  \notin Y$). In other
words,%
\[
F^{+}\cap Y=F\cap Y\ \ \ \ \ \ \ \ \ \ \text{and}\ \ \ \ \ \ \ \ \ \ F^{-}\cap
Y=F\cap Y.
\]

\end{vershort}

\begin{verlong}
From $F^{+}=F\cup\left\{  f\left(  F\right)  \right\}  $, we obtain%
\[
F^{+}\cap Y=\left(  F\cup\left\{  f\left(  F\right)  \right\}  \right)  \cap
Y=\left(  F\cap Y\right)  \cup\underbrace{\left(  \left\{  f\left(  F\right)
\right\}  \cap Y\right)  }_{\substack{=\varnothing\\\text{(since }f\left(
F\right)  \notin Y\text{)}}}=F\cap Y.
\]
From $F^{-}=F\setminus\left\{  f\left(  F\right)  \right\}  $, we obtain%
\[
F^{-}\cap Y=\left(  F\setminus\left\{  f\left(  F\right)  \right\}  \right)
\cap Y=F\cap\underbrace{\left(  Y\setminus\left\{  f\left(  F\right)
\right\}  \right)  }_{\substack{=Y\\\text{(since }f\left(  F\right)  \notin
Y\text{)}}}=F\cap Y.
\]

\end{verlong}

Claim 1 shows that the set $F\cap Y$ is nonempty. Hence, the set $F^{-}\cap Y$
is nonempty (since $F^{-}\cap Y=F\cap Y$). In other words, the set $F^{-}$
intersects $Y$. Thus, $F^{-}$ is not a subset of $X$ (since the sets $X$ and
$Y$ are disjoint). Thus, $F^{-}\notin C_{X}$ (since any face of $C_{X}$ is a
subset of $X$). Moreover, $F^{-}=F\setminus\left\{  f\left(  F\right)
\right\}  \subseteq F$. Thus, from $F\in B$, we obtain $F^{-}\in B$ (since $B$
is a simplicial complex). Combining this with $F^{-}\notin C_{X}$, we obtain
\[
F^{-}\in B\setminus C_{X}=M.
\]

We have $F\subseteq F\cup\left\{  f\left(  F\right)  \right\}  =F^{+}$. Thus,
if we had $F^{+}\in C_{X}$, then we would obtain $F\in C_{X}$ (since $C_{X}$
is a simplicial complex), which would contradict $F\notin C_{X}$. Hence,
$F^{+}\notin C_{X}$.

Let us now show that $F^{+}\in B$. As in the proof of Claim 1, we can show
that $F\in\left\{  \text{bicliques}\right\}  $ or $F\in C_{Y}$. Thus, we are
in one of the following two cases:

\textit{Case 1:} We have $F\in\left\{  \text{bicliques}\right\}  $.

\textit{Case 2:} We have $F\in C_{Y}$.

\begin{vershort}
Let us consider Case 1. In this case, $F$ is a biclique. In other words,
$F=U\cup V$ for some nonempty subsets $U\subseteq X$ and $V\subseteq Y$
satisfying $U\ \mathbf{R}\ V$ (by the definition of a biclique). Consider
these $U$ and $V$. From Lemma \ref{lem.XuY1}, we obtain $F\cap X=U$ and $F\cap
Y=V$. Moreover, the set $U\cup\left\{  f\left(  F\right)  \right\}  $ is
clearly nonempty and satisfies $U\cup\left\{  f\left(  F\right)  \right\}
\subseteq X$ (since $U\subseteq X$ and $f\left(  F\right)  \in X$).
\end{vershort}

\begin{verlong}
Let us consider Case 1. In this case, we have $F\in\left\{  \text{bicliques}%
\right\}  $. In other words, $F$ is a biclique. In other words, $F=U\cup V$
for some nonempty subsets $U\subseteq X$ and $V\subseteq Y$ satisfying
$U\ \mathbf{R}\ V$ (by the definition of a biclique). Consider these $U$ and
$V$. From Lemma \ref{lem.XuY1}, we obtain $F\cap X=U$ and $F\cap Y=V$.
Moreover, the set $U\cup\left\{  f\left(  F\right)  \right\}  $ is nonempty
(since $U$ is nonempty) and satisfies $U\cup\left\{  f\left(  F\right)
\right\}  \subseteq X$ (since $U\subseteq X$ and $f\left(  F\right)  \in X$).
\end{verlong}

\begin{vershort}
But $f\left(  F\right)  $ is an $X$-neighbor of $F\cap Y$. In other words,
$f\left(  F\right)  $ is an $X$-neighbor of $V$ (since $F\cap Y=V$). Moreover,
each element of $U$ is an $X$-neighbor of $V$ (since $U\ \mathbf{R}\ V$).
Combining the preceding two sentences, we conclude that each element of
$U\cup\left\{  f\left(  F\right)  \right\}  $ is an $X$-neighbor of $V$. In
other words, $\left(  U\cup\left\{  f\left(  F\right)  \right\}  \right)
\ \mathbf{R}\ V$. Hence, the set $\left(  U\cup\left\{  f\left(  F\right)
\right\}  \right)  \cup V$ is a biclique (by the definition of a biclique,
since $U\cup\left\{  f\left(  F\right)  \right\}  \subseteq X$ and $V\subseteq
Y$ are nonempty subsets satisfying $\left(  U\cup\left\{  f\left(  F\right)
\right\}  \right)  \ \mathbf{R}\ V$). Now,
\begin{align*}
F^{+}  &  =\underbrace{F}_{=U\cup V}\cup\left\{  f\left(  F\right)  \right\}
=U\cup V\cup\left\{  f\left(  F\right)  \right\}  =\left(  U\cup\left\{
f\left(  F\right)  \right\}  \right)  \cup V\\
&  \in\left\{  \text{bicliques}\right\}  \ \ \ \ \ \ \ \ \ \ \left(
\text{since }\left(  U\cup\left\{  f\left(  F\right)  \right\}  \right)  \cup
V\text{ is a biclique}\right) \\
&  \subseteq B\ \ \ \ \ \ \ \ \ \ \left(  \text{since }B=\left\{
\text{bicliques}\right\}  \cup C_{X}\cup C_{Y}\right)  .
\end{align*}
Thus, we have proved $F^{+}\in B$ in Case 1.
\end{vershort}

\begin{verlong}
But $f\left(  F\right)  $ is an $X$-neighbor of $F\cap Y$. In other words,
$f\left(  F\right)  $ is an $X$-neighbor of $V$ (since $F\cap Y=V$). In other
words,%
\begin{equation}
f\left(  F\right)  \ R\ v\ \ \ \ \ \ \ \ \ \ \text{for each }v\in V.
\label{pf.thm.biclique.collapse.c2.pf.3}%
\end{equation}
Moreover,%
\begin{equation}
u\ R\ v\ \ \ \ \ \ \ \ \ \ \text{for each }u\in U\text{ and }v\in V
\label{pf.thm.biclique.collapse.c2.pf.4}%
\end{equation}
(since $U\ \mathbf{R}\ V$). Combining the previous two sentences, we conclude
that
\[
u\ R\ v\ \ \ \ \ \ \ \ \ \ \text{for each }u\in U\cup\left\{  f\left(
F\right)  \right\}  \text{ and }v\in V
\]
(indeed, this follows from (\ref{pf.thm.biclique.collapse.c2.pf.3}) if
$u=f\left(  F\right)  $, and otherwise follows from
(\ref{pf.thm.biclique.collapse.c2.pf.4}) because $u\in U\cup\left\{  f\left(
F\right)  \right\}  $ and $u\neq f\left(  F\right)  $ imply $u\in U$). In
other words, $\left(  U\cup\left\{  f\left(  F\right)  \right\}  \right)
\ \mathbf{R}\ V$. Hence, the set $\left(  U\cup\left\{  f\left(  F\right)
\right\}  \right)  \cup V$ is a biclique (by the definition of a biclique,
since $U\cup\left\{  f\left(  F\right)  \right\}  \subseteq X$ and $V\subseteq
Y$ are nonempty subsets satisfying $\left(  U\cup\left\{  f\left(  F\right)
\right\}  \right)  \ \mathbf{R}\ V$). Now,
\begin{align*}
F^{+}  &  =\underbrace{F}_{=U\cup V}\cup\left\{  f\left(  F\right)  \right\}
=U\cup V\cup\left\{  f\left(  F\right)  \right\}  =\left(  U\cup\left\{
f\left(  F\right)  \right\}  \right)  \cup V\\
&  \in\left\{  \text{bicliques}\right\}  \ \ \ \ \ \ \ \ \ \ \left(
\text{since }\left(  U\cup\left\{  f\left(  F\right)  \right\}  \right)  \cup
V\text{ is a biclique}\right) \\
&  \subseteq B\ \ \ \ \ \ \ \ \ \ \left(  \text{since }B=\left\{
\text{bicliques}\right\}  \cup C_{X}\cup C_{Y}\right)  .
\end{align*}
Thus, we have proved $F^{+}\in B$ in Case 1.
\end{verlong}

Let us now consider Case 2. In this case, we have $F\in C_{Y}$. Hence,
$F\subseteq Y$ (since $C_{Y}$ is a complex with ground set $Y$). Thus, $F\cap
Y=F$. Moreover, $F$ is nonempty (as we have already shown in the proof of
Claim 1). But $f\left(  F\right)  $ is an $X$-neighbor of $F\cap Y$. In other
words, $f\left(  F\right)  $ is an $X$-neighbor of $F$ (since $F\cap Y=F$). In
other words, $f\left(  F\right)  \in X$ and $\left\{  f\left(  F\right)
\right\}  \ \mathbf{R}\ F$ (by the definition of an $X$-neighbor). From
$f\left(  F\right)  \in X$, we obtain $\left\{  f\left(  F\right)  \right\}
\subseteq X$. Hence, $\left\{  f\left(  F\right)  \right\}  \cup F$ is a
biclique (by the definition of \textquotedblleft biclique\textquotedblright,
since $\left\{  f\left(  F\right)  \right\}  \subseteq X$ and $F\subseteq Y$
are nonempty subsets satisfying $\left\{  f\left(  F\right)  \right\}
\ \mathbf{R}\ F$). Now,%
\begin{align*}
F^{+}  &  =F\cup\left\{  f\left(  F\right)  \right\}  =\left\{  f\left(
F\right)  \right\}  \cup F\\
&  \in\left\{  \text{bicliques}\right\}  \ \ \ \ \ \ \ \ \ \ \left(
\text{since }\left\{  f\left(  F\right)  \right\}  \cup F\text{ is a
biclique}\right) \\
&  \subseteq B\ \ \ \ \ \ \ \ \ \ \left(  \text{since }B=\left\{
\text{bicliques}\right\}  \cup C_{X}\cup C_{Y}\right)  .
\end{align*}
Thus, we have proved $F^{+}\in B$ in Case 2.

Now we have proved $F^{+}\in B$ in both Cases 1 and 2. Thus, $F^{+}\in B$
always holds. Combined with $F^{+}\notin C_{X}$, this yields
\[
F^{+}\in B\setminus C_{X}=M.
\]

The definition of $f$ shows that $f\left(  F^{+}\right)  $ is the largest
$X$-neighbor of the set $F^{+}\cap Y$, whereas $f\left(  F\right)  $ is the
largest $X$-neighbor of the set $F\cap Y$. Since $F^{+}\cap Y=F\cap Y$, these
two descriptions of $f\left(  F^{+}\right)  $ and $f\left(  F\right)  $ are
identical, so we conclude that $f\left(  F^{+}\right)  =f\left(  F\right)  $.
Similarly, from $F^{-}\cap Y=F\cap Y$, we obtain $f\left(  F^{-}\right)
=f\left(  F\right)  $. Combining this with $f\left(  F^{+}\right)  =f\left(
F\right)  $, we obtain $f\left(  F^{+}\right)  =f\left(  F^{-}\right)
=f\left(  F\right)  $.

Altogether, we have now shown that $F^{+}\in M$ and $F^{-}\in M$ and $f\left(
F^{+}\right)  =f\left(  F^{-}\right)  =f\left(  F\right)  $. In view of
$F^{+}=F\cup\left\{  f\left(  F\right)  \right\}  $ and $F^{-}=F\setminus
\left\{  f\left(  F\right)  \right\}  $, we can rewrite this as follows: We
have $F\cup\left\{  f\left(  F\right)  \right\}  \in M$ and $F\setminus
\left\{  f\left(  F\right)  \right\}  \in M$ and%
\[
f\left(  F\cup\left\{  f\left(  F\right)  \right\}  \right)  =f\left(
F\setminus\left\{  f\left(  F\right)  \right\}  \right)  =f\left(  F\right)
.
\]
Thus, Claim 2 is proved.
\end{proof}

\begin{statement}
\textit{Claim 3:} For any $F\in M$ and any $G\in M$ satisfying $G\subseteq F$,
we have $f\left(  F\right)  \leq f\left(  G\right)  $.
\end{statement}

\begin{proof}
[Proof of Claim 3.]Let $F\in M$ and $G\in M$ satisfy $G\subseteq F$. From
$G\subseteq F$, we obtain $G\cap Y\subseteq F\cap Y$.

\begin{vershort}
Recall that $f\left(  F\right)  $ is the largest $X$-neighbor of the set
$F\cap Y$. Thus, $f\left(  F\right)  $ is an $X$-neighbor of $F\cap Y$. Hence,
$f\left(  F\right)  $ is an $X$-neighbor of $G\cap Y$ as well (since $G\cap
Y\subseteq F\cap Y$).
\end{vershort}

\begin{verlong}
Recall that $f\left(  F\right)  $ is the largest $X$-neighbor of the set
$F\cap Y$. Thus, $f\left(  F\right)  $ is an $X$-neighbor of $F\cap Y$. In
other words, $f\left(  F\right)  \in X$, and%
\begin{equation}
\text{each }y\in F\cap Y\text{ satisfies }f\left(  F\right)  \ R\ y.
\label{pf.thm.biclique.collapse.c3.pf.1}%
\end{equation}

Therefore, each $y\in G\cap Y$ satisfies $f\left(  F\right)  \ R\ y$ as well
(since $y\in G\cap Y$ entails $y\in G\cap Y\subseteq F\cap Y$, and thus
(\ref{pf.thm.biclique.collapse.c3.pf.1}) yields $f\left(  F\right)  \ R\ y$).
In other words, $f\left(  F\right)  $ is an $X$-neighbor of $G\cap Y$ (since
$f\left(  F\right)  \in X$).
\end{verlong}

But $f\left(  G\right)  $ is the \textbf{largest} $X$-neighbor of the set
$G\cap Y$ (by the definition of $f$). Hence, $f\left(  G\right)  \geq x$
whenever $x$ is any $X$-neighbor of $G\cap Y$. Applying this to $x=f\left(
F\right)  $, we obtain $f\left(  G\right)  \geq f\left(  F\right)  $ (since
$f\left(  F\right)  $ is an $X$-neighbor of $G\cap Y$). In other words,
$f\left(  F\right)  \leq f\left(  G\right)  $. This proves Claim 3.
\end{proof}

Now, we can apply both parts \textbf{(a)} and \textbf{(b)} of Lemma
\ref{lem.pairing2} to $\Delta=B$ and $W=X\cup Y$ (since Claim 2 shows that
condition \textbf{(C1)} in Lemma \ref{lem.pairing2} holds, whereas Claim 3
shows that condition \textbf{(C2)} holds). Thus, we conclude that the map%
\begin{align*}
\mu:M  &  \rightarrow M,\\
F  &  \mapsto%
\begin{cases}
F\setminus\left\{  f\left(  F\right)  \right\}  , & \text{if }f\left(
F\right)  \in F;\\
F\cup\left\{  f\left(  F\right)  \right\}  , & \text{if }f\left(  F\right)
\notin F
\end{cases}
\end{align*}
is well-defined (by Lemma \ref{lem.pairing2} \textbf{(a)}), and the pair
$\left(  M,\mu\right)  $ is a matching on $B$ (by Lemma \ref{lem.pairing2}
\textbf{(a)} again), and that this matching $\left(  M,\mu\right)  $ is
acyclic (by Lemma \ref{lem.pairing2} \textbf{(b)}). Hence, there exists an
acyclic matching $\left(  M,\mu\right)  $ on $B$ with $M=B\setminus C_{X}$.
Thus, Theorem \ref{thm.dmt.collapse} (applied to $\Delta=B$ and $\Gamma=C_{X}%
$) shows that $B\searrow C_{X}$. As we said, this completes the proof of
Theorem \ref{thm.biclique.collapse}.
\end{proof}

\section{Functoriality and Dowker's theorem}

In this section, we will prove Dowker's and Barmak's theorems
(Theorems \ref{thm.dowker} and \ref{thm.barmak}).
The hard work has already been done in the preceding sections:
In fact, Theorem \ref{thm.barmak} follows immediately from
Theorem \ref{thm.biclique.collapse} if the sets $X$ and $Y$ are disjoint.
The only thing that remains to be done is reducing the general case
to this disjoint case.
Roughly speaking, this can be done by ``renaming'' the elements of
$X$ and $Y$, after one shows (Lemma~\ref{lem.isomorphic.simple.homotopy.type}
below) that isomorphic simplicial complexes are simple-homotopy-equivalent.
(This basic fact does not seem to appear explicitly in the literature,
so we prove it below -- using Theorem \ref{thm.biclique.collapse}
in fact!)

We shall now present this argument in some detail.
We begin by introducing some categories:

\begin{enumerate}

\item The \emph{category of simplicial complexes}:
The objects of this category are the simplicial complexes.
Its morphisms are the \emph{simplicial maps}, defined as
follows:
The \emph{minimal ground set} of a complex $\Delta$ is
defined to be the union of all faces of $\Delta$.
A \emph{simplicial map} from a complex $\Delta$ to a
complex $\Gamma$ means a map $f$ from the minimal ground
set of $\Delta$ to the minimal ground set of $\Gamma$
such that every face $F$ of $\Delta$ satisfies
$f\tup{F} \in \Gamma$ (where $f\tup{F}$ denotes
$\set{f\tup{x} \mid x \in F}$ as usual).
Composition of simplicial maps is just regular composition
of maps.

\item The \emph{category of relations}:
Given two relations $R \subseteq X \times Y$ and $S \subseteq Z \times W$,
a \emph{morphism of relations} $\phi \colon R \to S$ is a pair $\phi = (\phi_l,
\phi_r)$ of maps
$\phi_l \colon X \to Z$ and $\phi_r \colon Y \to W$ such that
each $(x,y) \in R$ satisfies $(\phi_l(x), \phi_r(y)) \in S$.
Thus, we can define the \emph{category of relations}: a category whose
objects are binary relations (more precisely, triples $\tup{X, Y, R}$
where $R \subseteq X \times Y$ is a relation), and whose morphisms are
morphisms of relations.
Composition of morphisms is defined componentwise
(i.e., by setting $\tup{\phi_l, \phi_r} \circ \tup{\psi_l, \psi_r}
= \tup{\phi_l \circ \psi_l, \phi_r \circ \psi_r}$),
and the identity morphism of a given relation $R \subseteq X \times Y$
is the pair $(\id_X, \id_Y)$.

\item The \emph{category of bipartite relations}:
Recall that a relation $R$ from a set $X$ to a set $Y$ is called a
bipartite relation if the sets $X$ and $Y$ are disjoint.
The \emph{category of bipartite relations} is the full subcategory
of the category of relations whose objects are the bipartite relations.

\end{enumerate}

Given a relation $R \subseteq X \times Y$, we let
$C_l(R) = C_X$ and $C_r(R) = C_Y$ be the left and right Dowker complexes of $R$.
% These are functorial in $R$, meaning that they form functors
% $C_l$ and $C_r$ from the category of relations to the category
% of simplicial complexes. (Their action on morphisms is
% defined in the obvious way: If $\phi = \tup{\phi_l, \phi_r}$, then
% $C_l(\phi)$ applies $\phi_l$ elementwise to each $Y$-conic set.)

If $R$ is a bipartite relation, then we furthermore let $B(R)$ be the
biclique complex of $R$ (formerly denoted $B$,
defined in Proposition \ref{prop.bicl-comp}).
% Thus, $B$ becomes a functor from the category of bipartite relations
% to the category of simplicial complexes.
% (Again, the action on morphisms is fairly clear.)

We have seen that $C_l(R) = C_X$ and $C_r(R) = C_Y$ are subcomplexes of $B(R) = B$.
We will write $i_l^R \colon C_l(R) \to B(R)$ and $i_r^R \colon C_r(R) \to B(R)$
for the respective inclusion maps.
% Topologically, these inclusion maps are strong
% deformation retracts by Theorem \ref{thm.biclique.collapse}
% (since the inclusion of a subcomplex $\Gamma$ into a complex $\Delta$
% satisfying $\Delta \searrow \Gamma$ is always a strong deformation retract).
% [DG] I'm not sure if we should call an inclusion map a strong deformation retract. But do we actually need to?

Given a morphism $\phi = (\phi_l, \phi_r)\colon R \to S$ of relations, the maps $\phi_l$ and $\phi_r$ induce
simplicial maps $C_l(\phi) \colon C_l(R) \to C_l(S)$ and $C_r(\phi) \colon C_r(R) \to C_r(S)$
of simplicial complexes.
(For instance, $C_l(\phi)$ replaces each vertex of $C_l(R)$ by its image under $\phi_l$.)
The assignments $R \mapsto C_l(R)$ and $R \mapsto C_r(R)$
thus define functors from the category of relations to the category of simplicial complexes.
If $R$ and $S$ are bipartite relations, then the morphism $\phi_l \cup \phi_r$ induces a
simplicial map $B(\phi) \colon B(R) \to B(S)$ of simplicial complexes.
This makes $R \mapsto B(R)$ a functor from the category of bipartite relations to the category
of simplicial complexes.
The inclusion maps $i_l^R \colon C_l(R) \to B(R)$ and $i_r^R \colon C_r(R) \to B(R)$
then form natural transformations of functors.

% $X' = X \times \{0\}$, let $Y' = Y \times \{1\}$ and let $R'$ be the
% subset of $X' \times Y'$ consisting of all pairs of the form $((x,0), (y,1))$
% with $(x,y) \in R$.
% Then $X'$ and $Y'$ are disjoint sets, so $R'$ is a bipartite graph.
% We let $B(R)$ be the biclique complex of $R'$. That is, $B(R)$ is the set 
% of subsets $U \cup V$ of $X' \cup Y'$ such that $U$ is a subset of $X'$,
% $V$ is a subset of $Y'$, and $(x,0) \in U$ and $(y,1) \in V$ implies
% $(x,y) \in R$. 
% We let $C_l(R)$ be the left Dowker complex of $R'$ and $C_r(R)$ be the right
% Dowker complex of $R'$.
% By construction $C_l(R)$ and $C_r(R)$ are simplicial subcomplexes of $B(R)$.
% We write $i_l^R \colon C_l(R) \to B(R)$ and $i_r^R \colon C_r(R) \to B(R)$
% for the inclusion maps.
% Note that by Theorem \ref{thm.biclique.collapse} the inclusions $i_l^R$ and $i_r^R$ 
% are strong deformation retracts.
% That is, $C_l(R)$ consists of all the $Y'$-conic sets in $X'$ and $C_r(R)$ consists
% of all the $X'$-conic sets in $Y'$.
% The function $X \to X'$, $x \mapsto (x,0)$ is a bijection, so $C_l(R)$ is
% isomorphic to $C_X$.
% Similarly, $C_Y(R)$ is isomorphic to $C_Y$.

As a last preparation for the proof of Theorem \ref{thm.barmak}, we now show that isomorphic
simplicial complexes are simple-homotopy-equivalent.

\begin{lemma}
  \label{lem.isomorphic.simple.homotopy.type}
  Let $\Delta$ and $\Delta'$ be isomorphic simplicial complexes. Then $\Delta$ and $\Delta'$ are
  simple-homotopy-equivalent.
\end{lemma}

\begin{proof}
\begin{vershort}
We WLOG assume that $\Delta$ and $\Delta^{\prime}$ are nonempty (since
otherwise, $\Delta=\Delta^{\prime}$, and the claim is obvious). Hence, the
empty set $\varnothing$ is both a face of $\Delta$ and a face of
$\Delta^{\prime}$.
\end{vershort}

\begin{verlong}
We WLOG assume that $\Delta$ and $\Delta^{\prime}$ are nonempty (since
otherwise, the isomorphy of $\Delta$ and $\Delta^{\prime}$ forces
\textbf{both} $\Delta$ and $\Delta^{\prime}$ to be empty, so that we have
$\Delta=\Delta^{\prime}$, and our claim is obvious). Hence, the empty set
$\varnothing$ is both a face of $\Delta$ and a face of $\Delta^{\prime}$.
\end{verlong}

Let $X$ and $X^{\prime}$ be the minimal ground sets of $\Delta$ and
$\Delta^{\prime}$. Let $\alpha:X\rightarrow X^{\prime}$ be the isomorphism
from $\Delta$ to $\Delta^{\prime}$. Let $\alpha_{\ast}:\Delta\rightarrow
\Delta^{\prime}$ be the map that sends each face $F$ of $\Delta$ to the face
$\alpha\left(  F\right)  =\left\{  \alpha\left(  f\right)  \ \mid\ f\in
F\right\}  $ of $\Delta^{\prime}$. This map $\alpha_{\ast}$ is a bijection
(since $\alpha$ is an isomorphism of complexes).

We refer to the elements of $X$ as the \emph{vertices} of $\Delta$. Likewise
for $X^{\prime}$ and $\Delta^{\prime}$.

Pick a set $Y$ that is disjoint from both $X$ and $X^{\prime}$ and has the
same size as $\Delta$. Pick a bijection $f\colon Y\rightarrow\Delta$. Thus,
$f$ assigns a face $f\left(  y\right)  $ of $\Delta$ to each $y\in Y$. Let
$f^{\prime}:Y\rightarrow\Delta^{\prime}$ be the composition $\alpha_{\ast
}\circ f$ of the bijections $f:Y\rightarrow\Delta$ and $\alpha_{\ast}%
:\Delta\rightarrow\Delta^{\prime}$.
Thus, $f^{\prime}$ is itself a bijection (since a composition of two
bijections is a bijection), and sends each $y\in Y$ to $\alpha_{\ast}\left(
f\left(  y\right)  \right)  =\alpha\left(  f\left(  y\right)  \right)  $ (by
the definition of $\alpha_{\ast}$).

Now, define a relation $R$ from $X$ to $Y$ by
\[
R=\left\{  \left(  x,y\right)  \in X\times Y\mid x\in f\left(  y\right)
\right\}  .
\]
Thus, a pair $\left(  x,y\right)  \in X\times Y$ belongs to $R$ if and only if
$x\in f\left(  y\right)  $. In other words, $R$ is the containment relation
between the vertices and the faces of $\Delta$, except that the faces have
been relabelled using $f$.

\begin{vershort}
The definition of the left Dowker complex $C_{l}\left(  R\right)  $ shows that%
\begin{align*}
C_{l}\left(  R\right)    & =\left\{  Y\text{-conic sets with respect to the
relation }R\right\}  \\
& =\left\{  U\subseteq X\ \mid\ U\text{ is empty or has a }Y\text{-neighbor}%
\right\}  \\
& =\left\{  U\subseteq X\ \mid\ U\text{ is empty or there exists some }y\in
Y\right.  \\
& \qquad\qquad\left.  \text{such that all }u\in U\text{ satisfy }%
u\ R\ y\right\}  \\
& \qquad\qquad\left(  \text{by the definition of a }Y\text{-neighbor}\right)
\\
& =\left\{  U\subseteq X\ \mid\ U\text{ is empty or there exists some }y\in
Y\right.  \\
& \qquad\qquad\left.  \text{such that all }u\in U\text{ satisfy }u\in f\left(
y\right)  \right\}  \\
& \qquad\qquad\left(  \text{by the definition of }R\right)  \\
& =\left\{  U\subseteq X\ \mid\ U\text{ is empty or there exists some }y\in
Y\right.  \\
& \qquad\qquad\left.  \text{such that }U\subseteq f\left(  y\right)  \right\}
\\
& =\left\{  U\subseteq X\ \mid\ U\text{ is empty or there exists some }%
F\in\Delta\right.  \\
& \qquad\qquad\left.  \text{such that }U\subseteq F\right\}  \qquad\left(
\text{since }f:Y\rightarrow\Delta\text{ is a bijection}\right)  \\
& =\left\{  U\subseteq X\ \mid\ U\text{ is empty or a subset of some face
}F\in\Delta\right\}  \\
& =\Delta
\end{align*}
(since $\varnothing$ is a face of $\Delta$, and since any subset of a face of
$\Delta$ is itself a face of $\Delta$).
\end{vershort}

\begin{verlong}
The definition of the left Dowker complex $C_{l}\left(  R\right)  $ shows that%
\[
C_{l}\left(  R\right)  =\left\{  Y\text{-conic sets with respect to the
relation }R\right\}  .
\]
Recall furthermore that $\Delta$ is a simplicial complex, and that
$\varnothing\in\Delta$ (since $\varnothing$ is a face of $\Delta$). Hence, it
is easy to see that
\[
C_{l}\left(  R\right)  \subseteq\Delta
\]
\footnote{\textit{Proof.} Let $U\in C_{l}\left(  R\right)  $. We shall prove
that $U\in\Delta$.
\par
If $U=\varnothing$, then this is obvious (since $\varnothing\in\Delta$). Thus,
we WLOG assume that $U\neq\varnothing$.
\par
{}We have
\[
U\in C_{l}\left(  R\right)  =\left\{  Y\text{-conic sets with respect to the
relation }R\right\}  .
\]
In other words, $U$ is a $Y$-conic set with respect to the relation $R$. In
other words, $U$ is a subset of $X$ and is empty or has a $Y$-neighbor (by the
definition of a $Y$-conic set). Since $U$ is not empty (because $U\neq
\varnothing$), we thus conclude that $U$ has a $Y$-neighbor. Let $y$ be this
$Y$-neighbor.
\par
Now, $y$ is a $Y$-neighbor of $U$. In other words, $y$ is an element of $Y$
such that all $u\in U$ satisfy $u\ R\ y$ (by the definition of a
$Y$-neighbor).
\par
Now, let $x\in U$ be arbitrary. Then, $x\ R\ y$ (since all $u\in U$ satisfy
$u\ R\ y$). In other words, $\left(  x,y\right)  \in R$. In other words, $x\in
f\left(  y\right)  $ (by the definition of the relation $R$).
\par
Forget that we fixed $x$. We thus have shown that $x\in f\left(  y\right)  $
for each $x\in U$. In other words, $U\subseteq f\left(  y\right)  $. But
$f\left(  y\right)  \in\Delta$. But $\Delta$ is a simplicial complex and
therefore closed under taking subsets. Hence, from $U\subseteq f\left(
y\right)  $ and $f\left(  y\right)  \in\Delta$, we obtain $U\in\Delta$ as
well.
\par
Forget that we fixed $U$. We thus have proved that $U\in\Delta$ for each $U\in
C_{l}\left(  R\right)  $. In other words, $C_{l}\left(  R\right)
\subseteq\Delta$. This completes our proof.} and%
\[
\Delta\subseteq C_{l}\left(  R\right)
\]
\footnote{\textit{Proof.} Let $U\in\Delta$. We shall show that $U\in
C_{l}\left(  R\right)  $.
\par
Define an element $y\in Y$ by $y=f^{-1}\left(  U\right)  $ (this is
well-defined, since $f:Y\rightarrow\Delta$ is a bijection). Then, $f\left(
y\right)  =U$. Hence, each $x\in U$ satisfies $x\in f\left(  y\right)  $
(since $x\in U=f\left(  y\right)  $) and therefore $\left(  x,y\right)  \in R$
(since the definition of the relation $R$ says that $\left(  x,y\right)  \in
R$ holds if and only if $x\in f\left(  y\right)  $) and thus $x\ R\ y$ (since
\textquotedblleft$x\ R\ y$\textquotedblright\ is just an alternative notation
for \textquotedblleft$\left(  x,y\right)  \in R$\textquotedblright). Renaming
the variable $x$ as $u$ in the preceding sentence, we can rewrite this as
follows: Each $u\in U$ satisfies $u\ R\ y$. Thus, $y$ is an element of $Y$
such that all $u\in U$ satisfy $u\ R\ y$. In other words, $y$ is a
$Y$-neighbor of $U$ (by the definition of a $Y$-neighbor). Therefore, the set
$U$ has a $Y$-neighbor (namely, $y$).
\par
Note that $U$ is a subset of $X$ (since $U\in\Delta$). Hence, $U$ is a subset
of $X$ and is empty or has a $Y$-neighbor (since $U$ has a $Y$-neighbor). In
other words, $U$ is a $Y$-conic set with respect to the relation $R$ (by the
definition of a $Y$-conic set). Thus,%
\[
U\in\left\{  Y\text{-conic sets with respect to the relation }R\right\}
=C_{l}\left(  R\right)  .
\]
\par
Forget that we fixed $U$. We thus have shown that $U\in C_{l}\left(  R\right)
$ for each $U\in\Delta$. In other words, $\Delta\subseteq C_{l}\left(
R\right)  $.}. Combining these two inclusions, we obtain $C_{l}\left(
R\right)  =\Delta$.
\end{verlong}

Next, define a relation $R^{\prime}$ from $X^{\prime}$ to $Y$ by
\[
R^{\prime}=\left\{  \left(  x^{\prime},y\right)  \in X^{\prime}\times Y\mid
x^{\prime}\in f^{\prime}\left(  y\right)  \right\}  .
\]
Thus, a pair $\left(  x^{\prime},y\right)  \in X^{\prime}\times Y$ belongs to
$R^{\prime}$ if and only if $x^{\prime}\in f^{\prime}\left(  y\right)  $. In
other words, $R^{\prime}$ is the containment relation between the vertices and
the faces of $\Delta^{\prime}$, except that the faces have been relabelled
using $f^{\prime}$.

We have already proved that $C_{l}\left(  R\right)  =\Delta$. An analogous
argument (using $X^{\prime}$, $\Delta^{\prime}$, $R^{\prime}$ and $f^{\prime}$
instead of $X$, $\Delta$, $R$ and $f$) shows that $C_{l}\left(  R^{\prime
}\right)  =\Delta^{\prime}$.

Next, we shall show that $C_{r}\left(  R\right)  =C_{r}\left(  R^{\prime
}\right)  $. Indeed, this follows easily from abstract nonsense: The pair
$\phi:=\left(  \alpha,\operatorname*{id}\nolimits_{Y}\right)  $ is an
isomorphism of relations from $R$ to $R^{\prime}$ (this follows easily from
the definitions of $R$ and $R^{\prime}$, since $f^{\prime}=\alpha_{\ast}\circ
f$ and since $\alpha$ is a bijection\footnote{In more detail: We need to show
that two elements $x\in X$ and $y\in Y$ satisfy $\left(  x,y\right)  \in R$ if
and only if they satisfy $\left(  \alpha\left(  x\right)  ,\operatorname*{id}%
\nolimits_{Y}\left(  y\right)  \right)  \in R^{\prime}$. But this follows from
the equivalences%
\begin{align*}
\left(  \left(  x,y\right)  \in R\right)  \   & \Longleftrightarrow\ \left(
x\in f\left(  y\right)  \right)  \qquad\left(  \text{by the definition of
}R\right)  \\
& \Longleftrightarrow\ \left(  \alpha\left(  x\right)  \in\alpha\left(
f\left(  y\right)  \right)  \right)  \qquad\left(  \text{since }\alpha\text{
is a bijection}\right)  \\
& \Longleftrightarrow\ \left(  \alpha\left(  x\right)  \in f^{\prime}\left(
y\right)  \right)  \qquad\left(  \text{since }\alpha\left(  f\left(  y\right)
\right)  =\alpha_{\ast}\left(  f\left(  y\right)  \right)  =f^{\prime}\left(
y\right)  \right)  \\
& \Longleftrightarrow\ \left(  \left(  \alpha\left(  x\right)  ,y\right)  \in
R^{\prime}\right)  \qquad\left(  \text{by the definition of }R^{\prime
}\right)  \\
& \Longleftrightarrow\ \left(  \left(  \alpha\left(  x\right)
,\operatorname*{id}\nolimits_{Y}\left(  y\right)  \right)  \in R^{\prime
}\right)  \qquad\left(  \text{since }y=\operatorname*{id}\nolimits_{Y}\left(
y\right)  \right)  .
\end{align*}
}). Hence, it induces an isomorphism of complexes $C_{r}\left(  \phi\right)
:C_{r}\left(  R\right)  \rightarrow C_{r}\left(  R^{\prime}\right)  $.
Recalling the definition of $C_{r}\left(  \phi\right)  $, we see that this
isomorphism $C_{r}\left(  \phi\right)  $ is simply the identity map (since it
replaces each vertex of $C_{r}\left(  R\right)  $ by its image under
$\operatorname*{id}\nolimits_{Y}$, which is exactly the same vertex). Thus,
the identity map is an isomorphism of complexes from $C_{r}\left(  R\right)  $
to $C_{r}\left(  R^{\prime}\right)  $. This shows that
\[
C_{r}\left(  R\right)  =C_{r}\left(  R^{\prime}\right)  .
\]

But $R$ is a bipartite relation (since the sets $X$ and $Y$ are disjoint).
Thus, Theorem \ref{thm.biclique.collapse} shows that its biclique complex
$B\left(  R\right)  $ collapses to both $C_{l}\left(  R\right)  $ and
$C_{r}\left(  R\right)  $. Therefore, the simplicial complexes $C_{l}\left(
R\right)  $ and $C_{r}\left(  R\right)  $ are simple-homotopy-equivalent. In
other words, $C_{l}\left(  R\right)  \sim C_{r}\left(  R\right)  $, where the
symbol $\sim$ means simple-homotopy-equivalence. Likewise, we obtain
$C_{l}\left(  R^{\prime}\right)  \sim C_{r}\left(  R^{\prime}\right)  $ (since
$R^{\prime}$ is also a bipartite relation). Since simple-homotopy-equivalence
is an equivalence relation, we can combine these results to%
\[
\Delta=C_{l}\left(  R\right)  \sim C_{r}\left(  R\right)  =C_{r}\left(
R^{\prime}\right)  \sim C_{l}\left(  R^{\prime}\right)  =\Delta^{\prime}.
\]
In other words, $\Delta$ and $\Delta^{\prime}$ are simple-homotopy-equivalent.
\end{proof}

Now we can easily obtain Theorem \ref{thm.barmak}:

\begin{proof}
[Proof of Theorem \ref{thm.barmak}.]If the sets $X$ and $Y$ are disjoint, then
Theorem \ref{thm.biclique.collapse} shows that the biclique complex $B$ of $R$
collapses to both $C_{X}$ and $C_{Y}$, and therefore the complexes $C_{X}$ and
$C_{Y}$ are simple-homotopy-equivalent. This proves Theorem \ref{thm.barmak}
in the case when the sets $X$ and $Y$ are disjoint.

Now let us consider the general case. We shall reduce this case to the
disjoint case by replacing $X$ and $Y$ with two disjoint copies.

Namely, define the two finite sets $\widetilde{X}=X\times\left\{  0\right\}  $
and $\widetilde{Y}=Y\times\left\{  1\right\}  $, which are clearly disjoint.
Let $\widetilde{R}$ be the subset of $\widetilde{X}\times\widetilde{Y}$
consisting of all pairs of the form $\left(  \left(  x,0\right)  ,\ \left(
y,1\right)  \right)  $ with $\left(  x,y\right)  \in R$. In other words,
$\widetilde{R}$ is the binary relation from $\widetilde{X}$ to $\widetilde{Y}$
for which two elements $\left(  x,0\right)  \in\widetilde{X}$ and $\left(
y,1\right)  \in\widetilde{Y}$ satisfy $\left(  x,0\right)  \ \widetilde{R}%
\ \left(  y,1\right)  $ if and only if $x\ R\ y$. Informally, this relation
$\widetilde{R}$ is simply \textquotedblleft the relation $R$ after each $x\in
X$ has been renamed as $\left(  x,0\right)  $ and each $y\in Y$ has been
renamed as $\left(  y,1\right)  $\textquotedblright. Rigorously, this means
that the relations $R$ and $\widetilde{R}$ are isomorphic; an explicit
isomorphism from $R$ to $\widetilde{R}$ is the pair $\phi:=\left(  \phi
_{l},\phi_{r}\right)  $, where $\phi_{l}:X\rightarrow\widetilde{X}$ is the map
sending each $x$ to $\left(  x,0\right)  $, and where $\phi_{r}:Y\rightarrow
\widetilde{Y}$ is the map sending each $y$ to $\left(  y,1\right)  $. This
isomorphism induces isomorphisms $C_{l}(\phi)\colon C_{l}(R)\rightarrow
C_{l}(\widetilde{R})$ and $C_{r}(\phi)\colon C_{r}(R)\rightarrow
C_{r}(\widetilde{R})$. Since isomorphic simplicial complexes are
simple-homotopy-equivalent (by Lemma \ref{lem.isomorphic.simple.homotopy.type}%
), this entails that $C_{l}(R)\sim C_{l}(\widetilde{R})$ and $C_{r}(R)\sim
C_{r}(\widetilde{R})$, where the symbol  $\sim$ means
simple-homotopy-equivalence. But the sets $\widetilde{X}$ and $\widetilde{Y}$
are disjoint, and thus $\widetilde{R}$ is a bipartite relation. Hence, we can
apply Theorem \ref{thm.barmak} to $\widetilde{X}$, $\widetilde{Y}$ and
$\widetilde{R}$ instead of $X$, $Y$ and $R$ (since we have already proved
Theorem \ref{thm.barmak} in the case when the sets $X$ and $Y$ are disjoint).
We conclude that $C_{l}(\widetilde{R})\sim C_{r}(\widetilde{R})$. Since
simple-homotopy-equivalence is an equivalence relation, we can now combine our
results:%
\[
C_{l}\left(  R\right)  \sim C_{l}(\widetilde{R})\sim C_{r}(\widetilde{R})\sim
C_{r}\left(  R\right)  .
\]
This proves Theorem \ref{thm.barmak} in the general case.
\end{proof}

% In the above proof we considered the simple homotopy equivalences
% \[
% C_{l}(R) \xrightarrow{C_l(\phi)} C_{l}(\widetilde{R})
% \xrightarrow{i_l^{\widetilde{R}}} B(\widetilde{R})
% \]
% and
% \[
% C_{r}(R) \xrightarrow{C_r(\phi)} C_{r}(\widetilde{R})
% \xrightarrow{i_r^{\widetilde{R}}} B(\widetilde{R}).
% \]
% These maps are all natural transformations between functors applied to the
% relation $R$. In this sense we have provided a sequence of natural simple
% homotopy equivalences between the right- and left Dowker complexes of the
% relation $R$.

\section{Further directions}

Dowker's paper \cite{Dowker52} has motivated much research in combinatorial
and applied topology over the 70 years since its publication (with 292
citations as of 2024).
Some of it generalizes and extends its results to other settings, or
approaches their proofs in different ways.
In this final section, we shall connect two such developments with our
Morse-theoretic approach.

\subsection{On the Brun--Salbu rectangle complex}

Brun and Salbu \cite{BruSal22} have recently given a new proof of
Theorem \ref{thm.dowker}, using projections instead of inclusions.
Their proof relies on the following notions (where $R$ is a bipartite
relation from $X$ to $Y$):

\begin{itemize}
\item A \emph{rectangle} of $R$ means a subset of $R$ having the form $U\times
V$ for some $U\subseteq X$ and some $V\subseteq Y$.

\item The \emph{rectangle complex} $E$ (called $E\left(  R\right)  $ in
\cite{BruSal22}) is defined to be the complex with ground set $X\times Y$ (not
$X\cup Y$) whose faces are the subsets of all rectangles of $R$.
\end{itemize}

Brun and Salbu show (\cite[Theorem 4.3]{BruSal22}) that this rectangle complex
$E$ is homotopy-equivalent to both $C_{X}$ and $C_{Y}$, but not via simplicial
embeddings like for $B$, but rather via simplicial projections. Namely, the
simplicial projection maps (defined by their action on the ground sets)%
\begin{align*}
\pi_{X}:E &  \rightarrow C_{X},\\
\left(  x,y\right)   &  \mapsto x
\end{align*}
and%
\begin{align*}
\pi_{Y}:E &  \rightarrow C_{Y},\\
\left(  x,y\right)   &  \mapsto y
\end{align*}
are shown to be homotopy equivalences (using Quillen's fiber lemma), so that
$C_{X}\simeq E\simeq C_{Y}$.

This result, too, can be proved using discrete Morse theory,
thus strengthening it to a simple-homotopy equivalence:

\begin{theorem}
\label{thm.E.she}The complexes $E$, $B$, $C_{X}$ and $C_{Y}$ are all
simple-homotopy-equivalent.
\end{theorem}

However, the only proof of this theorem we have found so far
is neither as elementary nor as simple as the above proof of
Theorem \ref{thm.barmak}.
It relies on the equivalence between simplicial complexes and
posets (via the barycentric subdivision -- see
\cite[proof of Proposition 9.29]{Kozlov20}) and on
simple-homotopy analogues of Quillen's fiber lemma and
the nerve theorem (cf. \cite{Barmak11}).
We intend to elaborate on it in forthcoming work.

\subsection{The relative case}

While Theorem \ref{thm.dowker} is commonly known as Dowker's theorem
nowadays, Dowker actually proved a different (if related) result in
\cite[Theorems 1 and 1a]{Dowker52}. Indeed, his result is simultaneously
weaker (making only homological rather than homotopical statements)
and stronger than Theorem \ref{thm.dowker}.
It is stronger, as it concerns itself with a more general situation:
a ``pair of relations'' (a relation $R_1$ with a subrelation $R_2$,
or, to use our functorial language, an injective morphism of relations
$i : R_2 \to R_1$) rather than a single relation.
The claim is then an isomorphism of relative (co)homology groups
between the respective Dowker complexes.
% The proof (in \cite[Sections 2--4]{Dowker52}) is 

\begin{question}
Can this generalized claim be proved using Morse-theoretical
methods?
\end{question}

This question is not quite straightforward, as
discrete Morse theory in relative homotopy is much less
well-trodden than in the absolute setting, and our
acyclic matching constructed implicitly in the proof of
Theorem \ref{thm.biclique.collapse} is not adapted to a
subrelation.

\subsection{Generalized Dowker Duality}

The idea of the Dowker nerve has been applied to concepts of relations between 
objects with more structure than discrete sets. 

For example in \cite{BrFoSa23}, the authors define a Dowker nerve for relations 
between categories, and prove a Dowker duality theorem for these.
The paper \cite{FerMin20} introduces the cylinder of a relation between partially
ordered sets.

Another direction is introduced in the paper \cite{FerMin20}.
Given a relation $R$ beween two sets $X$ and $Y$, we obtain a relation $D(R)$
between the left Dowker nerve $C_l(R) = C_X$ and the right Dowker nerve $C_r(R) = C_Y$
consisting of the pairs $(\sigma, \tau) \in C_l(R) \times C_r(R)$ of simplices such that
$\sigma \times \tau \subseteq R$. By \cite[Theorem 2.6]{FerMin20}, the cylinder
of $D(R)$ is simple-homotopy-equivalent to the face posets of $C_X$ and $C_Y$.
Since the barycentric subdivision of a simplicial complex is simple-homotopy-equivalent
to the complex itself, this implies that order complex of $D(R)$ is simple-homotopy-equivalent
to $C_X$ and $C_Y$.

% \begin{question}
% Is there a discrete Morse theory for surjective simplicial maps instead of subcomplexes?
% \end{question}

\begin{question}
Can the categorical version of Dowker's theorem \cite{BrFoSa23}  be proved
using discrete Morse theory for simplicial sets?
\end{question}

This would require a discrete Morse theory to be introduced for simplicial sets
in the first place. Such theories have been proposed for \textbf{semi}simplicial sets
previously, including in Brown's original paper \cite{Brown92} that laid the
ground for discrete Morse theory.

\subsection*{Acknowledgments}

The second author would like to thank Dmitry Kozlov and Tom Roby for
helpful and interesting conversations.


\bigskip

\begin{thebibliography}{99999999}                                                                                         %


\bibitem[Barmak11]{Barmak11}%
\href{https://doi.org/10.1016/j.jcta.2011.06.008}{Jonathan Ariel Barmak,
\textit{On Quillen's Theorem A for posets}, Journal of Combinatorial Theory,
Series A \textbf{118} (2011), pp. 2445--2453.}

\bibitem[Bjorne95]{Bjorne95}%
\href{https://publish.illinois.edu/ymb/files/2020/03/Bjorner_TopMeth.pdf}{Anders
Bj\"{o}rner, \textit{Topological Methods}, Chapter 34 of: R. Graham, M.
Gr\"{o}tschel, L. Lov\'{a}sz (eds), \textit{Handbook of Combinatorics},
Elsevier 1995.}

\bibitem[BrFoSa23]{BrFoSa23}\href{https://arxiv.org/abs/2303.16032v1}{Morten
Brun, Marius G\aa rdsmann Fosse, Lars M. Salbu, \textit{Dowker Duality for
Relations of Categories}, arXiv:2303.16032v1.}

\bibitem[Brown92]{Brown92}
Kenneth S. Brown,
\textit{The Geometry of Rewriting Systems: A Proof of the Anick-Groves-Squier Theorem}.
In: Algorithms and Classification in Combinatorial Group Theory,
Springer, 1992, pp. 137--163. \\
\url{https://pi.math.cornell.edu/~kbrown/scan/1992.0000.0137.pdf}

\bibitem[BruSal22]{BruSal22}%
\href{https://doi.org/10.1007/s00009-022-02213-0}{Morten Brun, Lars M. Salbu,
\textit{The Rectangle Complex of a Relation}, Mediterranean Journal of
Mathematics \textbf{20} (2023), article 7.} See
\href{https://arxiv.org/abs/2207.02018v2}{arXiv:2207.02018v2} for a preprint.

\bibitem[Dowker52]{Dowker52}\href{https://doi.org/10.2307/1969768}{C. H.
Dowker, \textit{Homology Groups of Relations}, Annals of Mathematics, Second
Series \textbf{56}, No. 1 (Jul., 1952), pp. 84--95.}

\bibitem[FerMin20]{FerMin20}%
  % \href{1234}
  {Fern{\'a}ndez, Ximena and Minian, El{\'\i}as Gabriel,
  \textit{
  The cylinder of a relation and generalized versions of the nerve theorem},
  Discrete \& Computational Geometry
  /textbf{63},
    (2020)}.

\bibitem[GrKaLe21]{GrKaLe21}\href{https://arxiv.org/abs/2102.07894v2}{Darij
Grinberg, Lukas Katth\"{a}n, Joel Brewster Lewis, \textit{The path-missing and
path-free complexes of a directed graph}, arXiv:2102.07894v2.}

\bibitem[Kozlov20]{Kozlov20}\href{https://bookstore.ams.org/gsm-207}{Dmitry N.
Kozlov, \textit{Organized Collapse: An Introduction to Discrete Morse Theory},
Graduate Studies in Mathematics \textbf{207}, AMS 2020}.

\bibitem[LinSha03]{LinSha03}%
\href{https://doi.org/10.1137/S0895480100366968}{Svante Linusson, John
Shareshian, \textit{Complexes of }$\mathit{t}$\textit{-colorable graphs}, SIAM
J. Discrete Math. \textbf{16} (2003), no. 3, pp. 371--389.}

\bibitem[Sagan19]{Sagan19}Bruce Sagan, \textit{Combinatorics: The Art of
Counting}, Graduate Studies in Mathematics \textbf{210}, AMS 2020.\newline%
\url{https://users.math.msu.edu/users/bsagan/Books/Aoc/final.pdf}\newline See
\url{https://users.math.msu.edu/users/bsagan/Books/Aoc/errata.pdf} for errata.
\end{thebibliography}
\end{document}